\documentclass[11pt]{amsart}
\usepackage{amssymb, amsmath, amscd, dcpic, pictexwd}
\usepackage{hyperref}

\usepackage{amsopn}

\usepackage{tikz-cd}

\usepackage{leftidx}

\DeclareMathOperator{\HW}{H-W}

\newtheorem{thm}{Theorem}[section]

\newtheorem{cor}[thm]{Corollary}
\newtheorem{lem}[thm]{Lemma}
\newtheorem{prop}[thm]{Proposition}
\newtheorem{defn}[thm]{Definition}
\newtheorem{rem}[thm]{Remark} 
\newtheorem{exm}[thm]{Example}
\newtheorem{con}[thm]{Conjecture}

\numberwithin{equation}{section}
\newcommand{\CC}{\mathbb C}
\newcommand{\PP}{\mathbb {P}}
\newcommand{\CA}{\mathbb{A}}
\newcommand{\Z}{\mathbb{Z}}
\newcommand{\Q}{\mathbb{Q}}
\newcommand{\LL}{\mathbb{L}}

\newcommand{\YY}{\mathcal{Y}}
\newcommand{\XX}{\mathfrak{X}}
\newcommand{\fg}{\mathfrak{g}}

\newcommand{\lw}{{\underline{w}}}

\newcommand{\OX}{\mathcal{O}}
\newcommand{\cR}{\mathcal{R}}

\newcommand{\DD}{\mathcal{D}}

\newcommand{\F}{\mathbb{F}}

\newcommand{\Fs}{{F_{\XX/S}}}
\newcommand{\Xp}{{\XX^{(p)}}}
\newcommand{\Yp}{{\YY^{(p)}}}
\newcommand{\homo}{{\mathcal{H}om}}

\newcommand{\pochf}[3]{\left(\frac{#1}{#2}\right)_{#3}}

\newcommand{\shyperg}[1]{{}_{n}F_{n-1}\left(\begin{array}{c}
	{\frac{1}{n+1}},{\frac{2}{n+1}},\cdots,{\frac{n}{n+1}} \\
	{1},{1},\cdots,{1}\end{array};{#1}\right)}
	
\DeclareMathOperator{\spec}{Spec}
\DeclareMathOperator{\Res}{Res}

\DeclareMathOperator{\Der}{Der}
\DeclareMathOperator{\Sol}{Sol}
\DeclareMathOperator{\Pic}{Pic}

\author{An Huang, Bong Lian, Shing-Tung Yau, Chenglong Yu}

\title[Hasse-Witt, unit roots and periods]{Hasse-Witt matrices, unit roots and period integrals}
\begin{document}
\maketitle
\begin{abstract}
Motivated by the work of Candelas, de la Ossa and Rodriguez-Villegas \cite{candelas}, we study the relations between Hasse-Witt matrices and period integrals of Calabi-Yau hypersurfaces in both toric varieties and partial flag varieties. We prove a conjecture by Vlasenko \cite{vlasenko} on higher Hasse-Witt matrices for toric hypersurfaces following Katz's method of local expansion \cite{Katz1984, Katz1985}. The higher Hasse-Witt matrices also have close relation with period integrals. The proof gives a way to pass from Katz's congruence relations in terms of expansion coefficients \cite{Katz1985} to Dwork's congruence relations \cite{Dwork2} about periods. 
\end{abstract}

\tableofcontents 
\baselineskip=16pt plus 1pt minus 1pt
\parskip=\baselineskip

\section{Introduction}
The relations among Hasse-Witt matrices, unit roots of zeta-functions and period integrals were pioneered in Dwork's work on the variation of zeta-functions of hypersurfaces \cite{Dwork1}, \cite{Dwork2}, \cite{Dwork3}. Some well-known examples are Legendre family
\begin{equation}
y^2=x(x-1)(x-\lambda),
\end{equation}
see Example 8 in \cite{Katz1}; and the Dwork family
\begin{equation}
X_0^{n+1}+\cdots X_n^{n+1}-(n+1)t X_0\cdots X_n=0,
\end{equation}
see 2.3.7.18 and 2.3.8 in \cite{Katz} and also \cite{jdyu}. There is a canonical choice of holomorphic $n$-forms $\omega_\lambda$ for these Calabi-Yau families since they are hypersurfaces in $\PP^n$. These families both have maximal unipotent monodromy at $\lambda=t^{-(n+1)}=0$. The period integral $I_\gamma$ of $\omega_\lambda$ over the invariant cycle $\gamma$ near $\lambda=0$ is the unique holomorphic solutions to the corresponding Picard-Fuchs equation. On the other hand, these families are defined over $\Z$. We can consider the $p$-reductions of these families and the Hasse-Witt matrices associated to $\omega_\lambda$. According to a theorem of Igusa-Manin-Katz, they are solutions to Picard-Fuchs equations mod $p$. We first state the relations for Dwork family, see \cite{jdyu}. The period is given by hypergeometric series
\begin{eqnarray}\label{F_n_n-1}
I_\gamma=F(\lambda) &=& \shyperg{\lambda} \\
	\nonumber
	&=& \sum_{r=0}^{\infty} \frac{\pochf{1}{n+1}{r}
		\pochf{2}{n+1}{r} \cdots\pochf{n}{n+1}{r}}{(r!)^n}
		\lambda^r.
\end{eqnarray}
The Hasse-Witt matrix $\HW_p$ is given by the truncation of $F(\lambda)$
\begin{equation}
\HW_p(\lambda)=\leftidx{^{(p-1)}} F(\lambda)=\sum_{r=0}^{p-1} \frac{\pochf{1}{n+1}{r}
		\pochf{2}{n+1}{r} \cdots\pochf{n}{n+1}{r}}{(r!)^n}
		\lambda^r.
\end{equation}
Here $\leftidx{^{(k)}} F(\lambda)$ means the truncation of $F(\lambda)$ with terms of $\lambda$ of degree less or equal than $k$.
Let \begin{equation}
g(\lambda)={F(\lambda)\over F(\lambda^p)}\in \mathbb{Z}_p[[\lambda]]
\end{equation}
Then $g$ is an element in $\varprojlim_{s\to \infty} \mathbb{Z}_p[\lambda,(\leftidx{^{(p-1)}} F(\lambda))^{-1}]/p^s\mathbb{Z}_p[\lambda,(\leftidx{^{(p-1)}} F(\lambda))^{-1}]$ and it satisfies Dwork congruences
\begin{equation}
g(\lambda)\equiv {\leftidx{^{(p^s-1)}} (F(\lambda))\over  \leftidx{^{(p^{s-1}-1)}} (F)(\lambda^p)}\mod p^s.
\end{equation}
Especially it is related to Hasse-Witt matrix by
\begin{equation}
{F(\lambda)\over F(\lambda^p)}\equiv \leftidx{^{(p-1)}} F(\lambda)\mod p.
\end{equation}
Let $q=p^r$ and $t\in \mathbb{F}_q$. Assume $p\nmid n+1$ and $t^{n+1}\neq 0, 1$ $\HW_p(\lambda)\neq 0$. Then there exists exactly one $p$-adic unit root in the factor of zeta function of Dwork family corresponding to Frobenius action on middle crystalline cohomology. It is given by 
\begin{equation}
g(\hat{\lambda})g(\hat{\lambda}^p)\cdots g(\hat{\lambda}^{p^{r-1}})
\end{equation}
with $\hat{\lambda}$ being the Teichm\"uller lifting under $\lambda\to \lambda^p$.

In this paper, we generalize the above relation to hypersurfaces in toric varieties and partial flag varieties. We first prove the mod $p$ results. Complete intersections are treated in Section \ref{complete}. The key algorithm of Hasse-Witt matrix is a generalization of the result on hypersurfaces in $\PP^n$. See Katz's algorithm 2.7 in \cite{Katz}. For general hypersurfaces in a Fano variety $X$, we use the Cartier operator on ambient space to localize the calculation in terms of local expansion similar to \cite{Katz1985}. When $X$ is toric variety, the algorithm depends on the toric data associated to $X$. The algorithm implies generic invertibility of Hasse-Witt matrices for toric hypersurfaces, generalizing Adolphson and Sperber's result for $\PP^n$ in \cite{adolphson2016}, see remark \ref{toricgeneral} and corollary \ref{toricinvert}. For generalized flag varieties, Bott-Samelson desingularization is used to reduce the calculation to a similar situation as toric varieties. The affine charts on Bott-Samelson varieties also give an explicit algorithm to calculate the power series expansions of period integrals of hypersurfaces in $G/P$.

The second part of the paper applies Katz's local expansion method \cite{Katz1984, Katz1985} to prove a conjecture in \cite{vlasenko}. The crystalline cohomology of the hypersurface family has an $F$-crystal structure. When the Hasse-Witt matrix is invertible, there exists a unit root part of the $F$-crystal. We consider the $p$-adic approximation of the Frobenius matrix on the unit root part. Especially, the Hasse-Witt matrix is the Frobenius matrix mod $p$. In \cite{Katz1985}, Katz gives a $p$-adic approximation of the Frobenius matrix in terms of the local expansions of top forms on a formal chart along a section of the family. In \cite{vlasenko}, Vlasenko constructed a sequence of matrices related to a Laurent polynomial $f$ and proved congruence relations similar to Katz's algorithm in \cite{Katz1985}. The $p$-adic limit is conjectured to be the Frobenius matrix for hypersurfaces in $\PP^n$ when $f$ is a homogenous polynomial. According to Corollary \ref{toricHW} in section \ref{local}, the first matrix $\alpha_1$ mod $p$ appeared in \cite{vlasenko} is the Hasse-Witt matrix for toric hypersurfaces . So it is natural to generalize Vlasenko's conjecture to toric hypersurfaces. We give a proof of the conjecture in section \ref{conj}.  

We first recall some notations for period integral
\subsection{Period Integral}
\label{period integral}
\begin{enumerate}
\item Let $X$ be a smooth semi-Fano variety of dimension $n$ over $\CC$. In this paper $X$ is toric variety or partial flag variety $G/P$ with $P$ parabolic subgroup in a semisimple algebraic group $G$.
\item We denote $V^\vee=H^0(X, \omega_X^{-1})$ to be the space of anticanonical sections.
\item For any nonzero section $s\in V^\vee$, the zero locus $Y_s$ is a Calabi-Yau hypersurface in $X$. 
\item Let $B$ be the set of $s\in V^\vee$ such that $Y_s$ is smooth. Then $B$ is Zariski open subset of $V^\vee$ and there is a family of smooth Calabi-Yau varieties $\pi\colon \YY\to B$ with $Y_s$ as fibers.
\item The section $s$ induces the adjunction formula $\omega_{Y_s}\cong \omega_X\otimes \omega_X^{-1}|_{Y_s}$. The constant function $1$ on the right hand side corresponds to a canonical section $\omega_s\in H^0(Y_s, \omega_{Y_s})$. In other words, the section $\omega_s$ is the residue of rational form ${1\over s}$. Putting $\omega_s$ together, we get a canonical section of $R^0\pi_*(\omega_{\YY/ B})$, denoted by $\omega$.
\end{enumerate}
\subsubsection{Period integral}
Consider the local system $\LL$ on $B$ formed by $H_{n-1}(Y_s, \Q)$, which is the dual of $R^{n-1}\pi_* \Q_\YY$. For any flat section $\gamma$ of $\LL$ on $U\subset B$ an open subset, the period integral $I_\gamma$ is defined by $\int_\gamma \omega$. 
\subsubsection{Picard-Fuchs system}
Let $\DD_{V^\vee}$ be the sheaf of linear differential operators generated by $\Der(V^\vee)$. Then any section $D$ of $\DD_{V^\vee}$ acts on the de-Rham coholomogy sheaf $R^{n-1} \pi_*(\Omega^\bullet(\YY/B))$ via restriction to $B$ and Gauss-Manin connection $\nabla$. Define the sheaf of Picard-Fuchs system for $\omega$ to be $PF(U)=\{{D}\in \DD_{V^\vee}(U)|\nabla({D}|_{U\cap B})\omega=0\}$. Period integrals are solutions to Picard-Fuchs system. In other words, we have ${D} I_\gamma=0$ for any ${D}\in PF(U)$.
\subsubsection{Solution-rank-1 points}
Consider the classical solution sheaf $\Sol=Hom_{\DD_{V^\vee}}(\DD_{V^\vee}/PF, \OX_{V^\vee})$. The solution rank at a point $s\in V^\vee$ is defined to be the dimension of the stalk $\Sol_s$. We will consider the points in $V^\vee$ having solution rank $1$. 
\subsubsection{Special point $s_0$}
\label{special point}
There exist special solution-rank-1 points when $X$ is toric or $G/P$. The theorem we will state is for those special solution-rank-1 points. We characterize $s_0$ up to scaling in terms of its zero locus $Y_{s_0}$ as follows. If $X$ is toric variety, we consider the stratification of $X$ by the torus action. Then $Y_{s_0}$ is the union of toric invariant divisors. If $X=G/P$, we consider the stratification of $X$ by projected Richardson varieties, see \cite{Thomas}. Then $Y_{s_0}$ is the union of projected Richardson divisors. Especially, if $P$ is a Borel subgroup, the divisor $Y_{s_0}$ is the union of Schubert divisors and opposite Schubert divisors. They are proven to have solution rank 1 from GKZ systems and tautological systems by Huang-Lian-Zhu \cite{HLZ}. The unique solution at $s_0$ is realized as a period integral $I_{\gamma_0}$ over invariant cycle $\gamma_0$. The special point $s_0$ is known as the large complex structure limit in the moduli of toric hypersurfaces. We expect the same result holds for flag varieties.
\begin{exm}
Let $X=\PP^n$ with homogenous coordinate $[x_0, \cdots, x_n]$. Then $V$ is identified with space of homogenous polynomials of degree $n+1$. In this case, the special solution-rank-1 point $s_0=x_0\cdots x_n$.
\end{exm}
\subsection{Hasse-Witt matrix}
Next we define the Hasse-Witt matrix. Let $k$ be a perfect field of characteristic $p$. Assume $\pi\colon Y\to S$ is a smooth family of Calabi-Yau variety over $k$ with relative dimension $n-1$. Let $\omega$ be a trivializing section of $R^0\pi_*(\omega_{Y/S})$. Let $\omega^*$ be the dual section of $R^{n-1}\pi_*(\OX_Y)$. The $p$-th power endomorphism of $\OX_Y$ induces a $p$-semilinear map $R^{n-1}\pi_*(\OX_Y)\to R^{n-1}\pi_*(\OX_Y)$ sending $\omega^*$ to $a \omega^*$. Then $\HW_p=a$ as a section of $\OX_S$ is the Hasse-Witt matrix under the basis $\omega^*$. The choice of $\omega$ for Calabi-Yau hypersurfaces is made by adjunction formula similar to period integral.  
%Assume $X$ has an smooth integral model outside finitely many primes. In other words, we have $\Pi \colon X_R\to \spec(R)$ being smooth over $R= {\Z}[{1\over M}]$ and $X_\CC\cong X$. Assume we have a basis of $R^0\Pi_*\omega^{-1}_{X_R/ \spec(R)}$ as free $R$-module $V$, after inverting finitely many primes in $R$, and denote the basis by $s_0\cdots s_N$. Then the universal family $\YY \to B$ has an integral model. The reduction to prime $p$ is $\pi\colon \YY_{\F_p}\to B_{\F_p}$. Let $\omega_p$ be the canonical trivializing section of $R^0\pi_*(\omega_{\YY_{\F_p}/ B_{\F_p}})$. It is constructed the same way as characteristic zero case using adjunction formula. The dual section in $R^{n-1}\pi_*(\OX_\mathcal{Y})$ is denoted by $\omega_p^*$. The $p$-th power endomorphism of $\OX_\YY$ induces a $p$-semilinear map $R^{n-1}\pi_*(\OX_\mathcal{Y})\to R^{n-1}\pi_*(\OX_\mathcal{Y})$ sending $\omega^*$ to $a \omega^*$. Then Hasse-Witt matrix $\HW_p$ for this family module $p$ is defined to be the function $a$ on $B_{\F_p}$. According to a result by Igusa, Manin and Katz (Proposition 2.3.6.3 in \cite{Katz}), Hasse-Witt matrix $\HW_p$ are solutions to Picard-Fuchs systems for $\omega_p$ in characteristic $p$.
\subsection{Statement of the theorem}
Now we state our main theorem. When $X$ is toric or $G/P$, it has an integral model over $\Z$. Let $s_0\in H^0(X, K_X^{-1})$ be the special solution-rank-1 point chosen in section \ref{special point}. Then $s_0$ can be extended as a basis $s_0\cdots s_N$ of $H^0(X, K_X^{-1})$. Let $a_0\cdots a_N$ be the dual basis for $s_0,\cdots s_N$. Suitable choices of $s_0\cdots s_N$ are still basis considering the $p$-reduction of $X$. See section \ref{local} and section \ref{flagv} for the details of choice of $s_i$ and the integration cycle $\gamma_0$. There exists the following truncation relation between Hasse-Witt invariants of hypersurfaces over $\F_p$ and period integrals. It can also be viewed as a relation between mod $p$ solutions to Picard-Fuchs systems and solution over $\CC$.
\begin{thm}
\label{main}
\begin{enumerate}
\item The Hasse-Witt matrix $\HW_p$ defined above are polynomials of $a_I$ of degree $p-1$. 
\item The period integral $I_{\gamma_0}$ defined above can be extended as holomorphic functions at $s_0$ and has the form ${1\over a_0} P({a_I\over a_0})$, where $P({a_I\over a_0})$ is a Taylor series of ${a_{1}\over a_0}, \cdots, {a_{N}\over a_0}$ with integer coefficients. 
\item They satisfy the following truncation relation
\begin{equation}
{1\over a_0^{p-1}} \HW_p= \leftidx{^{(p-1)}} P({a_I\over a_0}) \mod p
\end{equation}
where $\leftidx{^{(p-1)}} P({a_I\over a_0})$ is the truncation of $P$ at degree $p-1$.
\end{enumerate}
\end{thm}
%In other words, we can collect the solutions to Picard-Fuchs system module large enough primes $p$ to recover the solution in characteristic zero. Or more intuitively, we can write
%\begin{equation}
%\lim_{p\to \infty}{1\over a_0^p}\HW_p=I_\gamma.
%\end{equation}
%Since $\HW_p$ is related to counting $\F_p$-points module $p$ for $Y_s$ with $s\in B_{\F_p}(\F_p)$, a corollary of this result is a relation between counting points and period integrals, which was discussed by Candelas et al for special families.
%\begin{rem}
%Here the statement of the theorem a priori depends on the choice of the basis of sections for $H^0(X, \omega_X^{-1})$. However, the universal section $\omega$ does not depend on the choice of basis. Hence the Hasse-Witt invariant and period integral are both independent of the choice of basis. So the above statement should be proved in a way only depending on the choice of $s_0\in H^0(X, \omega_X^{-1})$. In the proof, we will see this is indeed the case.
%\end{rem}
\begin{rem}
In characteristic $p$, the conjugate spectral sequence provides a horizontal filtration for relative de-Rham cohomology. In particular, the Hasse-Witt matrix gives part of the coordinate for the projection of $R^{n-1}\pi_*(\OX_\mathcal{Y})$ to the horizontal subbundle in de-Rham cohomology. This is how Katz \cite{Katz} proved elements in Hasse-Witt matrix satisfy Picard-Fuchs equations. On the other hand, period integrals give the coordinate of horizontal sections in relative de-Rham cohomology over $\CC$. The above relations suggest that the horizontal subbundle provided by conjugate spectral sequence can approximate the horizontal section in characteristic zero near some degeneration point when $p\to \infty$.
\end{rem}

\begin{rem}
The Hasse-Witt matrices count rational point on Calabi-Yau hypersurfaces mod $p$ by Fulton's fixed point formula \cite{Fulton}. In the case of Calabi-Yau hypersurfaces in toric varieties, the relation between point counting and period integrals has been studied by Candelas, de la Ossa and Rodriguez-Villegas \cite{candelas}.

\end{rem}

Next we state the relation between period integrals and unit roots of zeta-function of toric hypersurfaces. Let $a_0=1$. The formal power series
\begin{equation}
g(a_I)={P(a_I)\over P(a_I^p)}
\end{equation}
lies in $\varprojlim_{s\to \infty} \mathbb{Z}_p[a_{I_1}\cdots a_{I_N}, (\leftidx{^{(p-1)}} P({a_I}))^{-1}]/p^s\mathbb{Z}_p[a_{I_1}\cdots a_{I_N},(\leftidx{^{(p-1)}} P({a_I}))^{-1}]$ and satisfies Dwork congruences
\begin{equation}
g(a_I)\equiv {\leftidx{^{(p^s-1)}} (P({a_I}))\over  \leftidx{^{(p^{s-1}-1)}} (P)(({a_I})^p)}\mod p^s.
\end{equation}
\begin{thm}
Let $a_I\in \mathbb{F}_q$. Assume the hypersurface $Y$ defined by $\sum a_I s_I$ is smooth and $\HW_p(a_I)\neq 0$. Then there exists exactly one $p$-adic unit root in the factor of zeta function of $Y$ corresponding to Frobenius action on $H^{n-1}_{cris}(Y)$. It is given by the formula
\begin{equation}
g(\hat{a}_I)g(\hat{a}_I^p)\cdots g(\hat{a}_I^{p^{r-1}})
\end{equation}
with $\hat{a}_I$ being the Teichm\"uller lifting under $a_I\to a_I^p$.
\end{thm}
Similar unit root formulas for general-type toric hypersurfaces and Calabi-Yau hypersurfaces in $G/P$ is given in \ref{conj} and \ref{flagF}.

\subsection{Acknowledgment} The authors are grateful to Mao Sheng, Zijian Yao, Dingxin Zhang, Jie Zhou for their interests and helpful discussions. 

\section{Local expansions and Hasse-Witt matrices}
\label{local}
Now we prove a algorithm of calculating Hasse-Witt matrices of hypersurfaces of $X$ in terms of local expansions of the sections. The key ingredient is to related the Hasse-Witt operator of Calabi-Yau or general type families $\YY$ to the Cartier operator on $X$. Then we apply the algorithm to toric and generalized flag varieties. Especially this recovers the algorithm for $\PP^n$.

We make the following assumptions for this section.
\begin{enumerate}
\item Let $X^n$ be a smooth projective variety defined over $k$ and satisfies $H^n(X, \OX)=H^{n-1}(X, \OX)=0$
\item Let $L$ be an base point free line bundle on $X$ and $V^\vee=H^0(X, L)\neq 0$ and $W^\vee=H^0(X,L\otimes K_X)\neq 0$. Let $a_1^\vee, a_2^\vee\cdots a_N^\vee$ and $e_1^\vee \cdots e_r^\vee$ be basis of $V^\vee$ and $W^\vee$.
\item Consider the smooth hypersurfaces $\YY$ over $S\hookrightarrow V^\vee-\{0\}$. Let $\XX$ be $X\times S$ and $i\colon \YY\to \XX$ be the embedding. The projections to $S$ are denoted by $\pi$
\item Let $F_S$ be the absolute Frobenius on $S$ and $\Xp=\XX\times_{F_S} S$ the fiber product. Then we have absolute Frobenius $F_\XX$the relative Frobenius $\Fs \colon X\to \Xp$. Denote $W\colon \Xp\to \XX$ and $\pi^{(p)}\colon \Xp\to S$ to be the projections. The corresponding diagram for family $\YY$ is defined in a similar way.
\end{enumerate}

Consider the following diagram
\begin{equation}
\label{exactsequence}
\begin{tikzcd}
0\arrow{r}&W^*L^{-1} \arrow{r}{f}\arrow{d}{f^{p-1}} &\OX_\Xp \arrow{r}\arrow{d}{F} &i_*\OX_{\Yp} \arrow{r}\arrow{d}{F} &0 \\
0\arrow{r}& {\Fs}_* L^{-1}\arrow{r}{f} &\Fs_*\OX_\XX\arrow{r} &i_*{F_{\YY/S}}_*\OX_\YY\arrow{r} &0
\end{tikzcd}
\end{equation}
The map $f^{p-1}\colon W^*L^{-1}\to {\Fs}_* L^{-1}$ is induced by ${\Fs}^*W^*L^{-1}=F_\XX^*L^{-1}\cong L^{-p}$ multiplied by $f^{p-1}$. 
This induces the diagram
\begin{equation}
\begin{tikzcd}
R^{n-1}\pi^{(p)}_*(\OX_{\YY^{(p)}})\arrow{r}\arrow{d}{F} & R^n\pi^{(p)}_*(W^*L^{-1})\arrow{d}{f^{p-1}}\\
R^{n-1}\pi_*({F_{\YY/S}}_*\OX_\YY)\arrow{r} & R^n\pi_*(L^{-1})
\end{tikzcd}
\end{equation}
The two horizontal maps are isomorphism. The left vertical map is the Hasse-Witt operator $\HW\colon F_S^*(R^{n-1}\pi_*(\OX_\YY))\cong R^{n-1}\pi^{(p)}_*(\OX_{\YY^{(p)}})\to R^{n-1}\pi^{(p)}_*({F_{\YY/S}}_*\OX_\YY)\cong R^{n-1}\pi_*\OX_\YY$. 

\begin{defn}
The basis $e_1^\vee \cdots e_r^\vee$ of $H^0(X, K_X\otimes L)$ induces a basis of $R^0\pi^*(\omega_{\YY/S})$ by residue map and dual basis $e_1\cdots e_r$ of $R^{n-1}\pi_*(\OX_\YY)$ under Serre duality. The Hasse-Witt matrix $a_{ij}$ is defined by $\HW(F_S^*(e_i))=\sum_j a_{ij}e_j$.
\end{defn}

Let $C_{X/S}\colon \omega_{\XX/S}\to \omega_{\Xp/S}$ be the top Cartier operator. For any coherent sheaf $M$ on $\XX$, the Grothendieck duality 
\begin{equation}
\Fs_*{\homo}(M,  \omega_{\XX/S})\cong \homo(\Fs_*M,  \omega_{\Xp/S})
\end{equation}
is related to $C_{X/S}$ by the natural pairing
\begin{equation}
\Fs_*{\homo}(M,  \omega_{\XX/S})\otimes \Fs_*M\to \omega_{\Xp/S}
\end{equation}
sending $g\otimes m$ to $C_{X/S}(g(m))$. Consider $M= L^{-p}$. Since $\Fs_* L^{-p}\cong W^* L^{-1}$, we have
\begin{equation}
\Fs_*{\homo}(L^{-p},  \omega_{\XX/S})\cong \homo(W^* L^{-1},  \omega_{\Xp/S})
\end{equation}
Then we have a morphism induced by multiplication by $f^{p-1}$
\begin{equation}
\Fs_*{\homo}(L^{-1},  \omega_{\XX/S})\to \Fs_*{\homo}(L^{-p},  \omega_{\XX/S})\cong \homo(W^* L^{-1},  \omega_{\Xp/S}).
\end{equation}
After taking $R^0\pi^{(p)}$ on both sides, we have an morphism
\begin{equation}
R^0\pi_*{\homo}(L^{-1},  \omega_{\XX/S})\to R^0\pi_*\homo(W^* L^{-1},  \omega_{\Xp/S}).
\end{equation}
This the dual of 
\begin{equation}
R^n\pi^{(p)}_*(W^*L^{-1})\to R^n\pi_*(L^{-1})
\end{equation}
Now we can conclude the dual of Hasse-Witt matrix is given by the following algorithm. Let $(t_1, \cdots, t_n)$ be local coordinate of $X$ at a point $x$. 
Denote $g(t)=\sum a_I t^I$ to be a formal power series. Then define $\tau(g)=\sum_J a_I t^J$ with $I=(p-1, \cdots, p-1)+pJ$. Fix a local trivialization section $\xi$ of $L$ on an Zariski open section $U$ containing $x$. Then any section ${e_i^\vee\over \xi}$ is a section of $\omega_X|_U$ and has the form $h_i(t)dt_1\wedge  dt_2 \cdots \wedge dt_n$. Under the same trivilization, the section ${e_i^\vee f^{p-1}\over \xi^p}$ has the form as $g_i(t)dt_1\wedge  dt_2 \cdots \wedge dt_n$. Then we claim $\tau(g_i)$ has the form $\tau(g_i)=\sum_j a_{ji} h_j$. 
\begin{thm}
\label{localHW}
The matrix $a_{ij}$ defined above is the Hasse-Witt matrix under the basis $e_1^\vee \cdots e_r^\vee$. 
\end{thm}
Notice that $\tau$ is $p^{-1}$-semilinear $\tau(h^p g)=h \tau(g)$. So the same algorithm works if we use a rational section $\xi$.

Now we specialize this algorithm to toric hypersurfaces or Calabi-Yau hypersurfaces. Let $X$ be a smooth complete toric variety defined by a fan $\sigma$. The $1$-dimensional primitive vectors $v_1,\cdots v_N$ correspond to toric divisors $D_i$. Assume $L=\OX(\sum a_i D_i)$ with $a_i\geq 1$. Let $\Delta=\{v\in \mathbb{R}^n|\langle v, v_i\rangle\geq -a_i\}$ and $\mathring{\Delta}$ the interior of $\Delta$. Then $H^0(X, L)$ has a basis corresponding to $u_I\in \Delta \cap \Z^n$ and $H^0(X, L\otimes K_X)$ has basis $e_i^\vee$ identified with $u_i\in \mathring{\Delta} \cap \Z^n$. Let $f=\sum a_I t^{u_I}$ be the Laurent series representing the universal section of $H^0(X, L)$ and $f^{p-1}=\sum A_{u} t^u$. Then we have 
\begin{cor}
\label{toricHW}
The Hasse-Witt matrix of hypersurface family over $|L|$ under the basis $e_i^\vee\in \mathring{\Delta} \cap \Z^n$ is given by $a_{ij}=A_{p u_j-u_i}$.
\end{cor}
\begin{proof}
After an action of $SL(n,\Z)$, we can assume $v_1\cdots v_n$ is the standard basis of $\Z^n$. Then $t_1,\cdots, t_n$ is an affine chart on $X$. We choose a section of $L\otimes K_X$ to be $s_0$ corresponding to origin in $\mathring{\Delta} \cap \Z^n$ and a meromorphic section of $K_X$ to be $\theta={dt_1\wedge  dt_2 \cdots \wedge dt_n\over t_1\cdots t_n}$. Let $s={s_0\over \theta}$ be a meromorphic section of $L$. Then
\begin{equation}
{e_i^\vee\over s}={t^{u_i}\over t_1\cdots t_n}dt_1\wedge  dt_2 \cdots \wedge dt_n.
\end{equation}
If we view $f$ as a section in $H^0(X, L)$, then
\begin{equation}
{f\over s}=t_1\cdots t_n \sum_I a_It^{u_I}
\end{equation}
Hence ${e_i^\vee f^{p-1}\over s^p}=g_i dt_1\wedge  dt_2 \cdots \wedge dt_n$ with
\begin{equation}
g_i={t^{u_i} (\sum_I a_It^{u_I})^{p-1}\over t_1\cdots t_n}=\sum_u A_ut^{u+u_i-\mathbf{1}}.
\end{equation}
Here $\mathbf{1}=(1,\cdots, 1)$. So 
\begin{equation}
\tau(g_i)=\sum_v A_u t^v
\end{equation}
with $u+u_i-\mathbf{1}=pv+(p-1)\mathbf{1}$.
On the other hand, we have 
\begin{equation}
\tau(g_i)=\sum_j a_{ji}t^{u_j-\mathbf{1}}.
\end{equation}
So $a_{ij}=A_{p u_j-u_i}$.
\end{proof}

\begin{rem}
When $X$ is $\PP^n$, Corollary \ref{toricHW} gives the same algorithm as Katz \cite{Katz}. In \cite{vlasenko}, Vlasenko defines the higher Hasse-Witt matrices for Laurent polynomial $f$. When $f$ is a homogenous polynomial of degree $d$, the first matrix $\alpha_1$ in \cite{vlasenko} mod $p$ is the Hasse-Witt matrix for hypersurface $Y$ in $\PP^n$. The $p$-adic limit of the matrices is conjectured to give the Frobenius matrix of the unit root part of $H^{n-1}_{cris}(Y)$, which is a dual analogue of matrices by Katz \cite{Katz1}. Corollary \ref{toricHW} proves that $\alpha_1$ mod $p$ is also the Hasse-Witt matrix for toric hypersurfaces. Hence it is natural to generalize Vlasenko's conjecture to toric hypersurfaces.
\end{rem}

If $X$ is any smooth variety satisfying the assumptions in this section and $L=K_X^{-1}$, then we have a Calabi-Yau family. In this case, the algorithm coincides with the criterion for Frobenius splitting of $X$ respect to $Y$. The basis of $H^0(X, L\otimes K_X)$ is chosen to be constant function $1$. The Hasse-Witt matrix is a function $a$ on $S$. We can choose the trivializing section of $L$ to be $(dt_1\wedge  dt_2 \cdots \wedge dt_n)^{-1}$. The local algorithm in this case it the following. 
\begin{cor}
\label{cyHW}
Let $f=g(t)(dt_1\wedge  dt_2 \cdots \wedge dt_n)^{-1}$. Then the Hasse-Witt matrix $a$ is given by $\tau(g^{p-1})$. More explicitly $a$ is the coefficient of $(t_1\cdots t_n)^{p-1}$ in local expansion of $(g(t))^{p-1}$.
\end{cor}
\begin{rem}
For any closed point $s\in S(k)$, the corresponding section $f_s^{p-1}\in H^0(X, \omega_X^{1-p})$ determines a Frobenius splitting of $X$ compatibly with $Y_s$ if and only if $a(s)\neq 0$. It is also equivalent to $Y_s$ being Frobenius split. Especially, Corollary \ref{toricHW} implies the well-known fact that toric variety $X$ is Frobenius split compatibly with torus invariant divisors. See Chapter 1 of \cite{Brion}. 
\end{rem}
\begin{proof}[Proof of Theorem \ref{main} for toric $X$] 
Following the previous notations, let $X$ be smooth complete toric variety and $L=K_X^{-1}=\OX_X(\sum_i D_i)$. Then the basis of $H^0(X, L)$ is identified with the integral points $u_I$ in the polytope $\Delta=\{v\in \mathbb{R}^n|\langle v, v_i\rangle\geq -1\}$. The universal section $f(t)=\sum a_I t^{u_I}$ with $u_0=(0,\cdots, 0)$. Then $\HW_p$ is the coefficient of constant term in $f^{p-1}$ according to Corollary \ref{toricHW} or \ref{cyHW}. On the other hand, the period integral 
\begin{equation}
I_\gamma={1\over (2\pi \sqrt{-1})^n}\int_\gamma {dt_1\wedge \cdots \wedge dt_n\over t_1\cdots t_n f(t)}
\end{equation}
along the cycle $\gamma\colon |t_1|=|t_2|=\cdots |t_n|=1$ is the coefficient of constant term in the Laurent expansion of $f^{-1}$. So
\begin{equation}
I_\gamma={1\over a_0}(1+\sum_{k=1}^\infty (-1)^k \sum_{k_1u_{I_1}+\cdots +k_lu_{I_l}=0, \sum k_j=k, I_j\neq 0} {k \choose k_1,k_2,\cdots, k_l}({a_{I_1}\over a_0})^{k_1}\cdots ({a_{I_l}\over a_0})^{k_l})
\end{equation}
and 
\begin{equation}
\HW_p=a_0^p(1+\sum_{k=1}^{p-1} \sum_{k_1u_{I_1}+\cdots +k_lu_{I_l}=0, \sum k_j=k, I_j\neq 0} {p-1\choose k_1,k_2,\cdots, k_l, p-1-k}({a_{I_1}\over a_0})^{k_1}\cdots ({a_{I_l}\over a_0})^{k_l})
\end{equation}
Then apply the congruence relation
\begin{equation}
{p-1 \choose k_1,k_2,\cdots, k_l, p-1-k}\equiv (-1)^k{k \choose k_1,k_2,\cdots, k_l}\mod p
\end{equation}
in the two expansions to get the conclusion.
\end{proof}

\begin{rem}[General toric hypersurfaces]
\label{toricgeneral}
The same argument also applies to general-type toric hypersurfaces. The entries in Hasse-Witt matrix are truncations of period integrals. The results for hypersurfaces in $\PP^n$ are proved by Adolphson and Sperber in \cite{adolphson2016}. We follow the notations in Corollary \ref{toricHW}. The sections $s\in H^0(X, L\otimes K_X)$ determines a section of $R^{0}\pi_*(\YY,\omega_{\YY/ B})$ via residue map and we can define period integral of $s$ in a similar way as Calabi-Yau hypersurfaces. Let $s_0\in H^0(X, K_X^{-1})$ be the large complex structure limit point with zero locus equal to the union of $D_i$. Let $e^\vee_i$ be the basis of $H^0(X, L\otimes K_X)$ corresponding to $u_i\in \mathring{\Delta} \cap \Z^n$ and denote $s_i=s_0\otimes e_i^\vee \in H^0(X, L)$. Let $f=\sum a_I t^{u_I}$ be the universal section of $L$. In the Laurent series expression of $f$, the section $s_i$ defined above is identified with multi-index $t^{u_i}$. The period integral of $e_i^\vee$ along the cycle $\gamma\colon |t_1|=|t_2|=\cdots |t_n|=1$ near $s_j$ is given by 
\begin{equation}
I_{\gamma,i}={1\over (2\pi \sqrt{-1})^n}\int_\gamma {t^{u_i}dt_1\wedge \cdots \wedge dt_n\over t_1\cdots t_n f(t)}
\end{equation}
and it is equal to the coefficient of $t^{-u_i}$ in the Laurent expansion of $f^{-1}$. On the other hand, the $ij$th entry $a_{ij}$ of the Hasse-Witt matrix under the basis $e_1^\vee\cdots e_r^\vee$ is given by the coefficient of $t^{pu_j-u_i}$ in the Laurent expansion of $f^{p-1}$. So we have the following 
\begin{enumerate}
\item The function $a_{ij}$ on $S$ are polynomials of $a_I$ of degree $p-1$. 
\item The period integral $I_{\gamma, i}$ is a holomorphic functions at $s_j$ and has the form ${1\over a_j} P_i({a_I\over a_j})$, where $P_i({a_I\over a_j})$ is a Taylor series of ${a_{I}\over a_j}$ with integer coefficients. 
\item They satisfies the following truncation relation
\begin{equation}
{1\over a_j^{p-1}} a_{ij}= \leftidx{^{(p-1)}} (P_i({a_I\over a_j})) \mod p
\end{equation}
where $\leftidx{^{(p-1)}} (P_i({a_I\over a_j}))$ is the truncation of $P$ at degree $p-1$.
\end{enumerate}
Since the period integral of $$\omega_i={t^{u_i}dt_1\wedge \cdots \wedge dt_n\over t_1\cdots t_n f(t)}$$ satisfies the corresponding Gel'fand-Kapranov-Zelevinski hypergeometric differential system, the entries $a_{ij}$ of Hasse-Witt matrices are mod $p$ solutions to the same differential system. See \cite{adolphson2014} for mod $p$ solutions to general hypergeometric systems. In \cite{adolphson2016},  Adolphson and Sperber also proved the generic invertibility of Hasse-Witt matrices for hypersurfaces in $\PP^n$. Similar idea gives the same result for toric hypersurfaces.
\end{rem}

\begin{cor}
\label{toricinvert}
The Hasse-Witt matrices for generic smooth toric hypersurface are not degenerate. In other words, the determinant $\det (a_{ij})\neq 0$. 
\end{cor}
\begin{proof}
Consider the determinant of matrix $(B_{ij})=(\leftidx{^{(p-1)}} (P_i({a_I\over a_j})))=({1\over a_j^{p-1}} a_{ij})$. The entry $\leftidx{^{(p-1)}}(P_i({a_I\over a_j}))$ has the form
\begin{equation}
\sum_{k=0}^{p-1} (-1)^k \sum_{u_{I_1}+\cdots +u_{I_k}=(k+1)u_j-u_i} ({a_{I_1}\over a_j})\cdots ({a_{I_k}\over a_j}).
\end{equation}
The indices $I_l$ are not required to be distinct. The constant term in $B_{ij}=\delta_{ij}$. Now we prove the constant term in $\det B$ is $1$. Let $\epsilon$ be a permutation of $r$-elements. Assume 
\begin{equation}
{a_{I_1^1}\over a_1}\cdots {a_{I_{k_1}^1}\over a_1}\cdot {a_{I_1^2}\over a_2}\cdots {a_{I_{k_2}^2}\over a_2}\cdots {a_{I_1^r}\over a_r}\cdots {a_{I_{k_r}^r}\over a_r}
\end{equation}
be a constant term appearing in the product $B_{\epsilon(1)1}\cdots B_{\epsilon(r)r}$. Then all indices $I_m$ appearing in the numerator correspond to interior integer points $u_i\in \mathring{\Delta} \cap \Z^n$ and satiesfy
\begin{equation}
u_{I_1^i}+\cdots+u_{I_{k_i}^i}+u_{\epsilon(i)}=({k_i}+1)u_i.
\end{equation}
Consider the vertex $u_{l}$ of the convex polytope generated by all $u_i\in \mathring{\Delta} \cap \Z^n$. Since the convex expression for such $u_{l}$ is unique, the indices $I_1^l=\cdots={k_l}^l={\epsilon(l)}=l$. Hence other terms in the product does not involve $u_l$. We can delete the vertices and consider the convex polytope generated by the remaining $u_i$ and get $\epsilon(i)=i$ inductively. Then the only constant term is $1$.
\end{proof}

\section{Generalized flag vareities}
\label{flagv}
Now we prove similar proposition for generalized flag variety $X=G/P$ using Corollary \ref{cyHW}. There is a natural candidate for large complex structure limit in Calabi-Yau hypersurfaces family of $G/P$, which is the union of codimension one stratum of projections of Richardson varieties, denoted by $Y_0$. See \cite{Thomas} for the definition of $Y_0$ and \cite{HLZ} for indentification of $Y_0$ as solution rank 1 point of Picard-Fuchs system. In the proof of toric Calabi-Yau families, we only used the following fact. There is an affine chart $(t_1\cdots t_n)\in \mathbb{A}_\Z^n$ on $X_\Z$ with $s_0=t_1 \cdots t_n (dt_1\cdots dt_n)^{-1}$. So we expect that $Y_0=Y_1+\cdots +Y_n+W$,where $Y_1, \cdots, Y_n$ has complete intersections at some point $x$ and $Z$ is an effective divisor outside $x$. But this only happens in some special cases, for example, projective spaces and Grassmannian $G(2,4)$. In general, the projections of Richardson varieties are not intersecting transversely to one point.  We need the Bott-Samelson-Demazure-Hansen type resolution of projections of Richardson varieties to lift the anticanonical sections to rational anticanonical sections. This construction is also used in the proof of Frobenius splitting for projections of Richardson varieties, see \cite{Thomas}.

Now let $\psi\colon Z\to X$ be a proper birational morphism between smooth varieties $Z$ and $X$ over $k$. Let $\omega_Z\cong\psi^*\omega_X+E$, where $E$ is a Weil divisor supported on exceptional divisor. Then we have $\psi^*\omega_X^{-1}\cong\omega_Z^{-1}+E$ inducing the isomorphism 
\begin{equation}
\psi_*(\omega_Z^{-1}+E)\cong \psi_*(\psi^*\omega_X^{-1})\cong \omega_X^{-1}.
\end{equation}
This isomorphism is induced by pulling back anticanonical sections on $X$ to anticanonical sections on $Z$ with poles along $E$ and hence fits in the commutative diagram of sheaves
\begin{equation}
\begin{tikzcd}
F_*(\omega_X^{1-p})\arrow{r}\arrow{d}{\hat{\tau}} & \psi_*F_*(\omega_Z^{1-p}((p-1)E))\arrow{d}{\hat{\tau}}\\
\OX_X\arrow{r} &  \psi_*(\OX_Z(E))
\end{tikzcd}
\end{equation}
Here $\hat{\tau}$ is the same trace map induced by Cartier operator as follows. If $\sigma=\sum_I f_I t^I (dt^1\wedge \cdots \wedge dt^n)^{1-p}$ is a local section of $\omega_X^{1-p}$, then $\hat{\tau}(F_*(\sigma))=\sum_J f_I t^J$ with $I=(p-1, \cdots, p-1)+pJ$. The map $\hat{\tau}\colon F_*(\omega_Z^{1-p}((p-1)E))\to \OX_Z(E)$ is defined as follows $\hat{\tau}(F_*(\sigma))=\hat{\tau}(F_*({1\over \eta^p}\eta^p \sigma))=\hat{\tau}({1\over \eta}F_*(\eta^p\sigma))={1\over \eta}\hat{\tau}(F_*(\eta^p\sigma))$. Here $\eta$ is local defining section of $E$. Then $\eta^p\sigma$ is a holomorphic section of $\omega_Z^{1-p}$ and $\hat{\tau}(F_*(\eta^p\sigma))$ is defined the same as $X$.
After taking global sections, we reduce the calculation of Hasse-Witt matrix to $Z$. If we have a section $s_0\in H^0(X, \omega_Z^{-1}(E))$ with the desired property as toric case, then similar conclusion follows. Note that we need to take into account meromorphic sections. When $E$ is union of some coordinate hypersurfaces at a point $x$, the same formula for $\hat{\tau}$ in terms of Laurent expansion of $\sigma$ in local coordinates still applies.

Now we apply the discussion above to Bott-Samelson-Demazure-Hansen varieties. They arise as resolutions of singularities of Schubert varieties and projections of Richardson varieties. See \cite{Brion}, section 2, \cite{Brion2005}, or \cite{Thomas}. First we fix some notations. Let $G$ be simple complex Lie group with Lie algebra $\fg$. Fix upper Borel subgroup $B=B^+$ and lower Borel subgroup $B^-$of $G$. Denote the simple roots by $\alpha_{1}\cdots \alpha_{l}$.  Let $W$ be the Weyl group and $s_{i}\in W$ the simple reflection generated by $\alpha_i$. Let $w=s_{i_1}\cdots s_{i_n}$ be a reduced expression for $w\in W$ and we denote it by $\lw=(s_{i_1},\cdots, s_{i_n})$. Let $P_{i_j}$ be the minimal parabolic subgroup corresponding to simple root $\alpha_{i_j}$. Then the Bott-Samelson variety $Z_\lw$ is defined to be $P_{i_1}\times\cdots \times P_{i_l}/B^n$. Here the right action by $B^n$ is defined by $(p_1,\cdots, p_n)\cdot (b_1,\cdots, b_n)=(p_1b_1,b_1^{-1}p_2b_2, \cdots, b_{n-1}^{-1}p_nb_n)$. The image of $(p_1,\cdots, p_n)$ under the quotient map is denoted by $[p_1,\cdots, p_n]$. 

We now recall some basic properties of Bott-Samelson variety. The map $\psi_\lw\colon Z_\lw\to G/B$ defined by $[p_1,\cdots, p_n]\mapsto p_1\cdots p_n$ is a birational map to the Schubert variety $X_w=BwB/B$. Let $Z_{\lw(j)}=P_{i_1}\times\cdots \hat{P}_{i_j} \cdots \times P_{i_l}/B^{n-1}$ be a divisor of $Z_\lw$ via the embedding $[p_1,\cdots, \hat{p}_j, \cdots, p_n]\mapsto [p_1,\cdots, 1, \cdots, p_n]$. The boundary of $Z_\lw$ is defined to be $\partial Z_\lw=Z_{\lw(1)}+\cdots+Z_{\lw(n)}$. These components have normal crossing intersection at $[1, \cdots, 1]$. Let $L(\lambda)=G\times_B k_{-\lambda}$ be the equivariant line bundle on $G/B$ associated to character $\lambda$ and $L_\lw(\lambda)=\psi_\lw^*L(\lambda)$. Then $\omega_{Z_\lw}\cong \OX_{Z_\lw}(-\partial {Z_\lw})\otimes L_\lw(-\rho)$ with $\rho$ being the sum of fundamental weights. 

From the previous discussion for toric case, we are looking for a special section $s_0\in H^0(X, \omega_X^{-1})$ and suitable affine chart on $Z_\lw$. Since the Picard group of $G/B$ is generated by opposite Schubert divisors, we have a section $\sigma$ of $L(\rho)$ vanishing exactly along all opposite Schubert divisors. Let $\tilde{s}_0$ be the tensor product of $\psi_\lw^*\sigma$ and canonical section of $\OX_{Z_\lw}(\partial {Z_\lw})$. Then $\tilde{s}_0$ vanishes along $\partial {Z_\lw}$ and preimage of opposite Schubert divisors. Let $U_{i_j}^-$ be the negative unipotent root subgroup $P_{i_j}\cap U^-$. The natural map $U_{i_1}^-\times \cdots \times U_{i_n}^-\to Z_\lw$ gives an affine neighborhood of $[1,\cdots , 1]$ which is isomorphic to $\CA^n$ with coordinate $(t_1, \cdots, t_n)$. Then $Z_{\lw(j)}$ on this affine chart is defined by $t_j=0$. The image of this chart under $\psi$ is inside the opposite Schubert cell $C^{id}=B^{-}B/B$. So $\tilde{s}_0$ vanishes with simple zero along coordinate hyperplanes on this chart. After rescaling, we can take $\tilde{s}_0=t_1\cdots t_n (dt_1\wedge\cdots \wedge dt_n)^{-1}$. Let $W^P$ be the set of minimal representatives in cosets $W/W_P$ and $w_P$ be the longest element in $W^P$ with reduced expression $\lw_P$. Then $\psi\colon Z_{\lw_P}\to G/P$ is birational and it is an isomorphism restricted to $Z_{\lw_P}-\partial Z_{\lw_P}\to Bw_PB/B\to Bw_PP/P$.The exceptional divisor is supported on $\partial Z_{\lw_P}$. Next we identify $\tilde{s}_0$ with a special anticanonical form defined on $X=G/P$. Let $X^v_w=X^v\cap X_w$ be the intersection of Schubert variety $X_w=BwB/B$ with opposite Schubert variety $X^v=B^-vB/B$. The image in $X=G/P$ is denoted by $\Pi_w^v$. This forms a stratification of $X$. The codimension one strata form an anticanonical divisor $\Pi_1+\cdots +\Pi_s$. See \cite{Thomas}, section 3. (Note that our notations for Schubert variety and opposite Schubert variety are different from \cite{Thomas}.) So there is an anticanonical section $s_0$ vanishing to the first order along $Y_0=\Pi_1\cup\cdots \cup\Pi_s$.
\begin{lem}
\label{form}
The two anticanonical sections are related by $\tilde{s}_0=\psi^*(s_0)$ up to a rescaling.
\end{lem}
\begin{proof}
We compare the two divisors $(\tilde{s}_0)$ and $(\psi^*(s_0))$. Since $\sigma$ vanishes along opposite Schubert divisors on $G/B$, then $\tilde{s}_0$ vanishes along the preimage of opposite Schubert divisors under $\psi_{\lw}$ and $\partial Z_\lw$. Let $C_{\lw_P}$ be the Schubert cell and $C^{id}$ the opposite Schubert cell. The restriction of $\psi_{\lw_P}\colon Z_{\lw_P}-\partial Z_{\lw_P}\to C_{\lw_P}$ is an isomorphism. Let $D_i$ be divisors supported on $C_{\lw_P}-C^{id}$. Then $(\tilde{s}_0)=\sum_i \overline{\psi_{\lw_P}^{-1}(D_i)} +\sum_j Z_{\lw_P(j)}$. The restriction of projection $C_{\lw_P}\to G/P$ is also isomorphism on its image. The divisors $\Pi_j$ are exactly the complement of the image of $C_{\lw_P}\cap C^{id}$. So we have $(\psi^*(s_0))|_{\psi_{\lw_P}^{-1}(C_{\lw_P})}=\sum_i \psi_{\lw_P}^{-1}(D_i)$. The exceptional locus of $\psi$ is supported on $\partial_{Z_{\lw_P}}$. So $(\psi^*(s_0))=\sum_i \overline{\psi_{\lw_P}^{-1}(D_i)} +\sum_j n_j Z_{\lw_P(j)}$ as a meromorphic anticanonical section. Since $(\tilde{s}_0)$ and $(\psi^*(s_0))$ are linear equivalent, then $\sum_j Z_{\lw_P(j)}$ and $\sum_j n_j Z_{\lw_P(j)}$ are linear equivalent. On the other hand, the divisors $Z_{\lw_P(1)}\cdots Z_{\lw_P(n)}$ form a basis for $\Pic(Z_{\lw_P})$, see \cite{Brion} Excercise 3.1.E (3). So $n_j=1$. 
\end{proof}
So the Hasse-Witt invariants have similar expansion algorithm as toric case according to the discussion above. On the other hand, the period integral near $s_0$ can also be calculated by pulling-back to $Z_\omega$. The cycle $\gamma\colon |t_1|=|t_2|=\cdots |t_n|=1$ has nontrivial image in $H_n(X-Y_{s_0})$ since the integral of $\int_\gamma {1\over \tilde{s}_0}\neq 0$. This is the unique invariant cycle near $s_0$ since $\dim H_c^n(X-Y_0)=1$. According to Theorem 1.4 in \cite{HLZ}, the period integral $\int_{\psi_*\gamma}{1\over f}=\int_{\gamma}{\psi^*({1\over f})}$ is the unique holomorphic solution to the Picard-Fuchs system near $s_0$. So we proved Theorem \ref{main} for generalized flag variety $X=G/P$. Note that the basis of $H^0(X, \omega_X^{-1})$ including $s_0$ can also  be written down explicitly in terms of standard monomials, see \cite{Brion2003}. This method also gives a way to calculate the power series expansions of period integrals of hypersurfaces in $G/P$.

\begin{rem}
The anticanonical form $s_0$ appears in \cite{rietsch2008} and \cite{Thomas}. In \cite{rietsch2008}, the form $s_0$ is constructed on torus chart of the open Richardson cell $\cR_w^{id}=C^{id}\cap C_w$ and glued together by coordinate transformations. We use the construction in \cite{Thomas} that the complement of $\cR_w^{id}$ is an anticanonical divisor. Lemma \ref{form} proves that $\psi^*(s_0)=t_1\cdots t_n (dt_1\wedge\cdots \wedge dt_n)^{-1}$ on the affine coordinate of $Z_\omega$, which is the local formula on torus chart appeared in \cite{rietsch2008}. This gives an explanation of the footnote in section 7 of \cite{rietsch2008}. The cycle $\gamma$ appears in \cite{rietsch2012} 7.1 for complete flag variety $G/B$, in \cite{marsh2013} Theorem 4.2 for Grassmannians and in \cite{lam2017} 12.4 for general $G/P$.
\end{rem}
Now we give some explicit examples of the resolution and the anticanonical form $s_0$ under the resolution.
\begin{exm}
Let $X$ be Grassmannian $G(2,4)$. Then $X=G/P$ with $G=SL(4)$ and $P=\{
  \left( {\begin{array}{cccc}
   \star & \star &\star &\star  \\
 \star & \star &\star &\star  \\
 0& 0 &\star &\star \\
  0& 0 &\star &\star 
  \end{array} } \right)\}
$. The Weyl group is $S_4$ and $W_P=S_2\times S_2$. The element $w_p=(13)(24)=(23)(34)(12)(23)=s_2s_3s_1s_2$. So $Z_\lw=P_1\times P_2\times P_3\times P_4/B^4$ with $P_1=\{\left( {\begin{array}{cccc}
   \star & \star &\star &\star  \\
0& \star &\star &\star  \\
 0& t_1 &\star &\star \\
  0& 0 &0 &\star 
  \end{array} } \right)\}
$, $P_2=\{\left( {\begin{array}{cccc}
   \star & \star &\star &\star  \\
0& \star &\star &\star  \\
 0& 0 &\star &\star \\
  0& 0 &t_2 &\star 
  \end{array} } \right)\}
$, $P_3=\{\left( {\begin{array}{cccc}
   \star & \star &\star &\star  \\
t_3& \star &\star &\star  \\
 0& 0 &\star &\star \\
  0& 0 &0 &\star 
  \end{array} } \right)\}
$ and $P_4=\{\left( {\begin{array}{cccc}
   \star & \star &\star &\star  \\
0& \star &\star &\star  \\
 0& t_4 &\star &\star \\
  0& 0 &0 &\star 
  \end{array} } \right)\}
$. The largest Schubert cell is $\{\left( {\begin{array}{cccc}
   a & b &1 &0  \\
c& d &0 &1\\
 1& 0 &0 &0 \\
  0& 1 &0 &0 
  \end{array} } \right)\}P/P
$ with coordinates $(a,b,c,d)$. The affine coordinate $(t_1, \cdots, t_4)\in \mathbb{A}^4$ on $Z_{\lw_P}$ is 
$$\{[\left( {\begin{array}{cccc}
   1 & 0 &0 &0  \\
0& 1&0 &0 \\
 0& t_1 &1 &0 \\
  0& 0 &0 &1 
  \end{array} } \right),\left( {\begin{array}{cccc}
   1 & 0 &0 &0  \\
0& 1&0 &0 \\
 0& 0 &1 &0 \\
  0& 0 &t_2 &1 
  \end{array} } \right),\left( {\begin{array}{cccc}
   1 & 0 &0 &0  \\
t_3& 1&0 &0 \\
 0& 0 &1 &0 \\
  0& 0 &0 &1 
  \end{array} } \right),\left( {\begin{array}{cccc}
   1 & 0 &0 &0  \\
0& 1&0 &0 \\
 0& t_4&1 &0 \\
  0& 0 &0 &1 
  \end{array} } \right)]\}.
$$
So the map $\psi\colon Z_{\lw_P}\to X$ under these local charts is given by 
\begin{equation}
\label{G24}
a={1\over t_1t_3}, b=-{t_1 + t_4\over t_1t_2t_3t_4}, c={1\over t_1}, d=-{1\over t_1t_2} 
\end{equation}
Recall the anticanonical section $s_0$ in \cite{HLZ} is given in terms of standard monomials as follows.
Let $\left( {\begin{array}{cccc}
   a_{11} & a_{12}  &a_{13}  &a_{14}   \\
 a_{21} & a_{22}  &a_{23}  &a_{24}  
  \end{array} } \right)$ be the basis of any two plane. The Pl\"ucker coordinates $x_{ij}$ are the determinant of $i,j$ columns. The section $s_0=x_{12}x_{23}x_{34}x_{14}$. In coordinate of Schubert cell, we have $s_0=-ad(ad-bc)(da\wedge db\wedge dc \wedge dd)^{-1}$. A direct calculation using (\ref{G24}) shows that $\psi^*s_0=t_1t_2t_3t_4(dt_1dt_2dt_3dt_4)^{-1}$. The other sections of $H^0(X, L)$ can also be written as  homogenous polynomials of $x_{ij}$ of degree $4$. 
%An interesting fact is that the pull-back of universal section $\psi(f)=t_1t_2t_3t_4g(t)(dt_1dt_2dt_3dt_4)^{-1}$ where $g(t)$ is a Laurent polynomial of $(t_1, \cdots, t_4)$ and $g(t)$ defines a toric variety $X_\delta$. The anticanonical bundle of $X_\delta$  
\end{exm}

\begin{rem}
The proof for both toric and flag varieties only depends on the the following fact. There is a torus chart $(t_1, \cdots, t_n)$ on the complement of $Y_{s_0}$ with $s_0=t_1\cdots t_n (dt_1\wedge\cdots \wedge dt_n)^{-1}$ on the chart. So Theorem {main} can hold for more general ambient spaces $X$. This also implies the Frobenius splitting of $X$ compatibly with $Y_0$.
\end{rem}

\section{Complete intersections}
\label{complete}
We further discuss the algorithm for Hasse-Witt matrix for complete intersections. 
\begin{enumerate}
\item Let $X^n$ be a smooth projective variety defined over $k$. Let $L_1,\cdots L_s$ be line bundles on $X$ and $E=\oplus_i L_i$.  Assume the following vanishing conditions $H^i(X, K_X\otimes \wedge^{s-i} E)=H^{i-1}(X, K_X\otimes\wedge^{s-i} E)=0$ for $i=1,\cdots, s$.
\item Let $V_i^\vee=H^0(X, L_i)\neq 0$ and $W^\vee=H^0(X,\det E\otimes K_X)\neq 0$. We further assume the zero locus of a generic element of $V^\vee=H^0(X, E)$ is smooth with codimension $s$. Let $e_1^\vee \cdots e_r^\vee$ be a basis of $W^\vee$. Let $f_i$ be the universal section of $L_i$ and $f=(f_1, \cdots, f_s)$ be universal section of $E$.
\item Consider the family of complete intersections defined by $f$ over the smooth locus $S\hookrightarrow V^\vee-\{0\}$. Let $\XX$ be $\XX\times S$ and $i\colon \YY\to \XX$ is the embedding of universal family. The projections to $S$ are denoted by $\pi$.
\item Let $F_S$ be the absolute Frobenius on $S$ and $\Xp=\XX\times_{F_S} S$ the fiber product. Then we have absolute Frobenius $F_\XX$ the relative Frobenius $\Fs \colon X\to \Xp$. Denote $W\colon \Xp\to \XX$ and $\pi^{(p)}\colon \Xp\to S$ to be the projections. The corresponding diagram for family $\YY$ is defined in a similar way.
\end{enumerate}

We repeat the argument in the hypersurfaces using the Koszul resolution 
\begin{equation}
0\to \wedge^s E^\vee \to \wedge^{s-1} E^\vee \to \cdots \to E^\vee\to \OX_\XX\to i_*\OX_\YY\to 0.
\end{equation}
Standard spectral sequence argument together with the vanishing assumptions gives an isomorphism $R^{n-s}\pi_*(\OX_\YY)\to R^n\pi_*((\det E)^{-1})$. The maps in the Koszul resolution is given as follows. Identify the section of $\wedge^k E^\vee$ with the sections $(f^\vee_{j_1, \cdots, j_k})$ of $\oplus_{j_1,\cdots, j_k} L_{j_1}^\vee\otimes \cdots \otimes L_{j_k}^\vee$ with ordered set $j_1,\cdots, j_k$, such that $j_1,\cdots, j_k$ are distinct and $f_{j_1, \cdots, j_k}^\vee={\pm } f_{j_1^\prime, \cdots, j_k^\prime}^\vee$ if ${j_1, \cdots, j_k}$ is a permutation of ${j_1^\prime, \cdots, j_k^\prime}$ with signature $\pm 1$. Then $(f^\vee_{j_1, \cdots, j_{k+1}})$ is mapped to $(f^\vee_{j_1, \cdots, j_k})=(\sum_j f^\vee_{j_1, \cdots, j_{k},j}f_j)$. We have similar commutative diagram as (\ref{exactsequence}).
\begin{equation}
\begin{tikzcd}
0\arrow{r}&W^*\det E^\vee \arrow{r}{f}\arrow{d}{f^{p-1}} &\cdots &\OX_\Xp \arrow{r}\arrow{d}{F} &i_*\OX_{\Yp} \arrow{r}\arrow{d}{F} &0 \\
0\arrow{r}& {\Fs}_* \det E^\vee\arrow{r}{f} &\cdots &\Fs_*\OX_\XX\arrow{r} &i_*{F_{\YY/S}}_*\OX_\YY\arrow{r} &0
\end{tikzcd}
\end{equation}
The map $f^{p-1}\colon W^*\wedge^{k} E^\vee \to {\Fs}_* \wedge^{k} E^\vee$ is induced by multiplication $(f^\vee_{j_1, \cdots, j_k})\mapsto (({f^\vee_{j_1, \cdots, j_k}})^p f_{j_1}^{p-1}\cdots f_{j_k}^{p-1})$. So we have commutative diagram
\begin{equation}
\begin{tikzcd}
R^{n-s}\pi^{(p)}_*(\OX_{\YY^{(p)}})\arrow{r}\arrow{d}{F} & R^n\pi^{(p)}_*(W^*\det E^\vee)\arrow{d}{f^{p-1}}\\
R^{n-s}\pi_*({F_{\YY/S}}_*\OX_\YY)\arrow{r} & R^n\pi_*(\det E^\vee)
\end{tikzcd}
\end{equation}
with horizontal maps being isomorphisms. The left vertical map is the Hasse-Witt operator $\HW\colon F_S^*(R^{n-s}\pi_*(\OX_\YY))\cong R^{n-s}\pi^{(p)}_*(\OX_{\YY^{(p)}})\to R^{n-s}\pi^{(p)}_*({F_{\YY/S}}_*\OX_\YY)\cong R^{n-s}\pi_*\OX_\YY$. So we have similar definition of Hasse-Witt matrix under basis $e_1^\vee \cdots e_r^\vee$.
\begin{defn}
The basis $e_1^\vee \cdots e_r^\vee$ of $H^0(X, K_X\otimes \det E)$ induces a basis of $R^0\pi^*(\omega_{\YY/S})$ by residue map and dual basis $e_1\cdots e_r$ of $R^{n-s}\pi_*(\OX_\YY)$ under Serre duality. The Hasse-Witt matrix $a_{ij}$ is defined by $\HW(F_S^*(e_i))=\sum_j a_{ij}e_j$.
\end{defn}
The same argument in hypersurfaces case gives us the algorithm of computing Hasse-Witt matrix in terms of local expansion of $f$. Now $f^{p-1}$ is replaced by $(f_1\cdots f_s)^{p-1}$. Under the trivialization $\xi$ of $\det E$ under local coordinates $(t_1, \cdots, t_n)$,  the section ${e_i^\vee\over \xi}$ has the form $h_i(t)dt_1\wedge  dt_2 \cdots \wedge dt_n$ and ${e_i^\vee(f_1\cdots f_s)^{p-1}\over \xi^p}$ has the form as $g_i(t)dt_1\wedge  dt_2 \cdots \wedge dt_n$. Then $\tau(g_i)$ has the form $\tau(g_i)=\sum_j a_{ji} h_j.$ 

On the other hand, the period integral has the form $\int_\gamma \Res {\Omega\over f_1\cdots f_s}=\int_{\gamma^\prime} {\Omega\over f_1\cdots f_s}$ where $\gamma^\prime$ is a cycle in the complement of $\{f_1\cdots f_s=0\}$. So it is the same form as hypersurfaces with $f$ replaced by $f_1\cdots f_s$. The section $f_1\cdots f_s$ also defines a subfamily of hypersurfces in the linear system $|\det E|$. Both Hasse-Witt matrices and period integrals can be calculated with the same algorithm applied to this subfamily. So the truncation relation still holds for complete intersections in both toric variety and flag variety. For example, the statement for toric Calabi-Yau hypersurfaces is as follows. Let $X$ be a smooth toric variety and $K_X^{-1}=D_1+\cdots +D_s$ be a partition of toric invariant divisors. Let $L_i=\OX(D_i)$. Let $f_{ij}$ be a basis of $H^0(X, L_i)$ consisting of monomials with $f_{i0}$ the defining section of $D_i$. The universal section is $f=(\sum_j b_{ij}f_{ij})_i$. Then the period integral of the unique invariant cycle near $f_0=(f_{i0})$ has the form ${1\over b_{10}\cdots b_{s0}}P({b_{1j_1}\cdots b_{sj_s} \over b_{10}\cdots b_{s0}})$, in which $P({b_{1j_1}\cdots b_{sj_s} \over b_{10}\cdots b_{s0}})$ is a Taylor series of ${{b_{1j_1}\cdots b_{sj_s} \over b_{10}\cdots b_{s0}}}$ with integer coefficients. The degree-${(p-1)}$ truncation $\leftidx{^{(p-1)}}P({b_{ij}\over b_{i0}})$ multiplied by $(b_{10}\cdots b_{s0})^{p-1}$ is a degree-$(p-1)$ polynomial of $b_{1j_1}\cdots b_{sj_s}$ and gives the Hasse-Witt matrix for the Calabi-Yau complete intersection family.

\section{Frobenius matrices of toric hypersurfaces}
\label{conj}
Now we give a proof of the conjecture in \cite{vlasenko} for toric hypersurfaces. First we state the conjecture. The notations follow \cite{Katz1985}. Let $k$ be a perfect field of characteristic $p$. Let $W(k)$ be the ring of Witt vectors of $k$. Denote $\sigma\colon W\to W$ be the absolute Frobenius automorphism of $W$. For any $W$-scheme $Z$, let $Z_0=Z\otimes_W k$ be the reduction mod $p$. Let $S=\spec(R)$ be an affine $W$-scheme. Let ${R}_\infty=\varprojlim R/p^sR$ and $S_\infty=Spf(R_\infty)$. We fix a Frobenius lifting on $R$ and also denote it by $\sigma$, which is a ring endomorphism $\sigma\colon R\to R$ such that $\sigma(a)=a^p\mod pR$. Let $X$ be a smooth complete toric variety defined by a fan. The $1$-dimensional primitive vectors $v_1,\cdots v_N$ correspond to toric divisors $D_i$. Assume $L=\OX(\sum k_i D_i)$ with $k_i\geq 1$. Let $\Delta=\{v\in \mathbb{R}^n|\langle v, v_i\rangle\geq -k_i\}$ and $\mathring{\Delta}$ the interior of $\Delta$. Then $H^0(X, L)$ has a basis corresponding to $u_I\in \Delta \cap \Z^n$ and $H^0(X, L\otimes K_X)$ has basis $e_i^\vee$ identified with $u_i\in \mathring{\Delta} \cap \Z^n$.  Let $f=\sum a_I t^{u_I}, a_I\in R$ be a Laurent series representing a section of $H^0(X, L)$. Let $(\alpha_s)_{i,j}$ be a matrix with $ij$-th entry equal to the coefficient of $t^{p^su_j-u_i}$ in $(f(t))^{p^s-1}$. The endmorphism $\sigma$ is also extended entry-wisely to matrices. It is proved in \cite{vlasenko} that $\alpha_s$ satisfies the following congruence relations
\begin{thm}[Theorem 1 in \cite{vlasenko}]
\label{cong}
\label{conjecture}
\begin{enumerate}
\item For $s\geq 1$, $$\alpha_s\equiv \alpha_1\cdot \sigma(\alpha_1)\cdots \sigma^{s-1}(\alpha_1)\mod p.$$
\item Assume $\alpha_1$ is invertible in $R_\infty$. Then
$$ \alpha_{s+1}\cdot \sigma(\alpha_s)^{-1}\equiv \alpha_s\cdot \sigma(\alpha_{s-1})^{-1}\mod p^s.$$
\item Under the condition of (2), for any derivation $D\colon R\to R$, we have
$$D(\sigma^m(\alpha_{s+1}))\cdot \sigma^m(\alpha_{s+1})^{-1}\equiv D(\sigma^m(\alpha_{s}))\cdot \sigma^m(\alpha_{s})^{-1}\mod p^{s+m}.$$
\end{enumerate}
\end{thm}

Suppose that $f$ defines a smooth hypersurface $\pi\colon Y\to S$. We assume $Y$ satisfies the condition (HLF) in \cite{Katz1985}, the Hodge cohomology groups $H^j(Y, \Omega^i_{Y/S})$ are locally free $R$-modules for $i+j=n-1$. We also assume the pair $(X, Y)$ satisfies (HLF), the Hodge cohomology groups $H^j(X, \Omega_{X/S}^i(\log Y))$ are locally free $R$-modules for $i+j=n$. Consider the $F$-crystal $H^{n-1}_{cris}(Y_0/S_\infty)\cong H^{n-1}_{DR}(Y/S)\otimes_R R_\infty$. We further assume the family $Y/S$ satisfy condition $HW(n-1)$ in \cite{Katz1985}, which says for any $s_0\colon R_0\to K$ with $K$ perfect field, the Hasse-Witt operator $H^{n-1}(Y^{(p)}_{s_0},\OX_{Y^{(p)}_{s_0}})\to H^{n-1}(Y_{s_0},\OX_{Y_{s_0}})$ is an automorphism. Notice that $\alpha_1\mod p$ is the Hasse-Witt matrix under the dual basis of $\omega_i=\Res {t^{u_i}dt_1\wedge \cdots \wedge dt_n\over t_1\cdots t_n f(t)}\in H^0(Y, \Omega^{n-1}_{Y/S})$ according to Corollary \ref{toricHW}. The condition $HW(n-1)$ can be checked using $\alpha_1$. In particular, this condition implies that $\alpha_1\mod p$ is invertible. The unit-root $F$-crystal $U_0\subset H^{n-1}_{cris}(Y_0/S_\infty)$ and slope $\leq n-2$ sub-crytal $U_{\leq n-2}$ are defined under the assumption. The quotient $Q_{n-1}=H^{n-1}_{cris}/U_{\leq n-2}$ is isomorphic to the $p^{n-1}$-twist of the dual $U_0^\vee$ to $U_0$. The projection of $\omega_i$ to $Q_{n-1}$ gives a dual basis of $U_0$. In \cite{vlasenko}, the Frobenius matrix and connection matrix of $U_0$ are conjectured to be the limits of matrices in Theorem \ref{conjecture}.
\begin{con}[\cite{vlasenko}]
\label{conv}
The Frobenius matrix is the $p$-adic limit 
\begin{equation}
F=\lim_{s\to \infty} \alpha_{s+1}\sigma(\alpha_s)^{-1}.
\end{equation}
The connection matrix is given by 
\begin{equation}
\nabla_D=\lim_{s\to \infty} D(\alpha_{s})(\alpha_s)^{-1}.
\end{equation}
\end{con}
Now we give the proof of this conjecture under an additional assumption on $(X,L)$.
\begin{thm}
\label{main2}
Let $(X, L)$ be a smooth toric variety with line bundle $L=\OX(\sum k_i D_i)$. Let $p_i$ be toric invariant points corresponding to top dimensional cone in the fan decomposition. If a generic section of $L$ does not vanish at some $p_i$, then Conjecture $\ref{conv}$ is true.
\end{thm}
The assumption in the theorem can be checked from toric data, or replaced by the equivalent assumption on the polytope of $|L|$. Let $p_i$ be the intersection of $D_1\cdots D_n$. Under a transformation of $SL(n,\mathbb{Z})$, we can assume the corresponding cone is generated by standard basis of $\mathbb{R}^n$. Let $f=\sum_I a_It^{u_I}$ as before. Then under a trivialization of $L$, the universal section $f$ is $t_1^{k_1}\cdots t_n^{k_n}(\sum_I a_It^{u_I})$. So a generic section $f$ does not vanish at $(t_1, \cdots, t_n)=(0,\cdots,0)$ means $(-k_1,\cdots, -k_n)$ is a vertex of $\Delta$. The assumption the theorem is equivalent to that at least one of the vertices of $\Delta$ is the intersection of hyperplanes $\langle v, v_i\rangle= -k_i, 1\leq i\leq n$ with $v_i\cdots v_n$ generating a cone of $X$. This is satisfied by $X=\PP^n$ with $L=\OX(d), d\geq n+1$.
\begin{proof}
The proof follows the ideas in Katz's proof of Theorem 6.2 \cite{Katz1985}. Consider the $F$-crystal constructed by logarithmic crystalline cohomology $H^{n}_{cris}(X_0, Y_0)\cong H^n_{DR}(X, Y)\otimes R_\infty$. From the long exact sequence
\begin{equation}
\cdots \to H^n_{DR}(X)\to H^n_{DR}(X, Y)\to H^{n-1}_{DR}(Y)(-1)\to \cdots 
\end{equation}
and $H^k_{DR}(X)$ is concentrated in $H^{{k\over2}, {k\over2}}$, the corresponding subcrystal $U_{\leq n-1}$ and quotient $Q_{n}$ are also defined on $H^{n}_{cris}(X_0, Y_0)$ by taking the inverse image of $U_{\leq n-2}$ subcrystal in $H^{n-1}_{cris}(Y_0)$ and $$Q_n(H^{n}_{cris}(X_0, Y_0))\cong Q_{n-1}(H^{n-1}_{cris}(Y_0))(-1).$$
Here $H(-1)$ means the Frobenius action is multiplied by $p$. We also have isomorphism $H^{n}_{cris}(X_0, Y_0)=(H^0(X, \Omega^n_{X/S}(Y))\otimes R_\infty)\oplus U_{n-1}$. So we only need to consider the Frobenius matrix acting on projections of log $n$-forms $\omega_i= {t^{u_i}dt_1\wedge \cdots \wedge dt_n\over t_1\cdots t_n f(t)}$ onto $Q_n$.
We can assume the primitive vectors $v_1,\cdots v_n$ are the standard basis of $\mathbb{R}^n$. The cone generated by $v_1\cdots v_n$ defines an affine coordinate $(t_1, \cdots, t_n)$ on $X$ which is isomorphic to $\mathbb{A}^n$. First we assume $Y$ is away from $(t_1, \cdots, t_n)=0$ and consider the formal expansion map along $(t_1, \cdots, t_n)=0$
\begin{equation}
P\colon H^n_{DR}(X, Y)\otimes R_\infty\to H^n_{DR}(R_\infty[[t_1,\cdots, t_n]]/R_\infty).
\end{equation}
Similar as Katz's proof of Theorem 6.2 \cite{Katz1985}, we have the following conjecture
\begin{con}
\label{key}
$U_{\leq n-1}$ is the kernel of formal expansion map. 
\end{con}Actually a weaker statement that $U_{\leq n-1}$ is contained in the the kernel can imply that $U_{\leq n-1}$ is the kernel, see remark \ref{kernel}. The conjecture might be proved by log version of the theory of {de Rham}-Witt following Katz's proof. We will first give the proof of Theorem \ref{main2} assuming the conjecture and state the method to get around the conjecture at the end. Assume the local expansion of $1\over f$ exist in $R[[t_1,\cdots, t_n]][{1\over t_1}, \cdots, {1\over t_n}]$ and has the form
\begin{equation}
{1\over f}=\sum_u A_{u}t^u.
\end{equation}
Notice that $f$ may not have an inverse in $R[[t_1,\cdots, t_n]][{1\over t_1}, \cdots, {1\over t_n}]$. We can consider the localization of $R$ by inverting the coefficient $a_{u_0}$ of the vertex $u_0=(-k_1,\cdots, -k_n)$. Let $a_{u_0}=1$, then
\begin{equation}
{1\over f}={t^{-u_0}\over 1+\sum_{u_I\neq u_0} a_It^{u_I-u_0}}=t^{-u_0}(1+\sum_k(-1)^k(\sum_{u_I\neq u_0} a_It^{u_I-u_0})^k)=\sum_u A_{u}t^u.
\end{equation}
So the local expansion of $\omega_i$ has the form
\begin{equation}
\omega_i= {t^{u_i}dt_1\wedge \cdots \wedge dt_n\over t_1\cdots t_n f(t)}={dt_1\wedge \cdots \wedge dt_n\over t_1\cdots t_n}\cdot \sum_u A_ut^{u+u_i}
\end{equation}
with all $u+u_i>0$. Assume the Frobenius action on $\omega_i$ is in the form 
\begin{equation}
F(\omega_i^{(\sigma)})\equiv \sum_j f_{ij} \omega_j \mod U_{\leq n-1}
\end{equation}
and the connection of $\nabla_D$ has the form
\begin{equation}
\nabla(D)(\omega_i)\equiv \sum_j {\nabla(D)}_{ij} \omega_j \mod U_{\leq n-1}
\end{equation}
Assume Conjecture \ref{key} is true, then
\begin{equation}
\label{image1}
F(\omega_i^{(\sigma)})=\sum_j f_{ij} \omega_j  \text{  in  } H^n_{DR}(R_\infty[[t_1,\cdots, t_n]]/R_\infty)
\end{equation}
and 
\begin{equation}
\label{image2}
\nabla(D)(\omega_i)= \sum_j {\nabla(D)}_{ij} \omega_j \text{  in  } H^n_{DR}(R_\infty[[t_1\cdots t_n]]/R_\infty).
\end{equation}
According to the Frobenius action on $H^n_{DR}(R_\infty[[t_1,\cdots, t_n]]/R_\infty)$, we compare the coefficient of $t^{p^kv}$ for multi-index $v\in \mathbb{Z}^n$ and $v>0$
\begin{equation}
\label{congruence1}
p^n\sigma({A_{{u}_i^\prime}})\equiv  \sum_j f_{ij} A_{u_j^{\prime\prime}}\mod p^k
\end{equation}
with $p({{u}_i^\prime}+u_i)=p^kv=u_j^{\prime\prime}+u_j$. 
On the other hand, we compare the expansions of $f^{p^s-1}=\sum_u \tilde{A}^s_u t^u$ and $1\over f$
\begin{equation}
\sum_u \tilde{A}^s_u t^u=f^{p^s}{1\over f}= (\sum B^s_ut^u)(\sum_u A_ut^u).
\end{equation}
So $\tilde{A}^s_u=  \sum_{u^\prime+u^{\prime\prime}=u} B^s_{u^\prime}A_{u^{\prime\prime}}$. We can extend $\sigma$ to any Laurent series with coefficients in $R$ by $\sigma(t^u)=t^{pu}$. Since $\sigma(f)=f^p+pg$, then
\begin{equation}
\sigma(f^{p^{s-1}})=\sigma(f)^{p^{s-1}}=(f^p+pg)^{p^{s-1}}\equiv f^{p^s}\mod p^s.
\end{equation}
Let $f^{p^{s-1}-1}=\sum_u \tilde{A}^{s-1}_u t^u$, then
\begin{equation}
\sum_u \sigma(\tilde{A}^{s-1}_u) t^{pu}\equiv f^{p^s}\sigma({1\over f})= (\sum_{u}B^s_ut^u)(\sum_u \sigma (A_u)t^{pu})\mod p^s.
\end{equation}
So $\sigma(\tilde{A}^{s-1}_u)\equiv \sum_{u^\prime+pu^{\prime\prime}=pu} B^s_{u^\prime}\sigma(A_{u^{\prime\prime}})\mod p^s$. Now we compute $\sigma(\alpha_{s-1})_{im}=\sigma(\tilde{A}^{s-1}_{p^{s-1}u_m-u_i})$ in terms of $A_u$ and $B_u$. It is the sum of $B^s_{u^\prime}\sigma(A_{u^{\prime\prime}})$ with $u^\prime+pu^{\prime\prime}=p({p^{s-1}u_m-u_i})$. The factor $B^s_{u^\prime}$ is the sum of the terms
\begin{equation}
{p^s \choose k_1,k_2,\cdots, k_l}a_{I_1}^{k_1}\cdots a_{I_l}^{k_l}
\end{equation}
with $k_1u_{I_1}+\cdots+k_lu_{I_l}=u^\prime$. Denote $\nu_p$ to be the $p$-adic valuation. Let $k=\min\{\nu_p(k_1)\cdots \nu_p(k_l)\}$ and $p^kv=p^su_m-u^\prime=p(u^{\prime\prime}+u_i)$ in (\ref{congruence1}), then
\begin{equation}
\sigma(A_{u^{\prime\prime}})\equiv \sum_jf_{ij}A_{u^{\prime\prime}_j}\mod p^k
\end{equation}
with $u^\prime+u^{\prime\prime}_j=p^su_m-u_j$. Since the $p$-adic valuation of multinomial has estimate
\begin{equation}
\nu_p {p^s \choose k_1,k_2,\cdots, k_l}\geq s-\min\{\nu_p(k_1)\cdots \nu_p(k_l)\},
\end{equation}
then 
\begin{equation}
{p^s \choose k_1,k_2,\cdots, k_l}a_{I_1}^{k_1}\cdots a_{I_l}^{k_l}\sigma(A_{u^{\prime\prime}})\equiv \sum_jf_{ij}{p^s \choose k_1,k_2,\cdots, k_l}a_{I_1}^{k_1}\cdots a_{I_l}^{k_l}A_{u^{\prime\prime}_j}\mod p^s.
\end{equation}
So we have 
\begin{equation}
B_{u^\prime}\sigma(A_{u^{\prime\prime}})\equiv \sum_jf_{ij}B_{u^\prime}A_{u^{\prime\prime}_j}\mod p^s
\end{equation}
with $u^\prime+u^{\prime\prime}=p^{s-1}u_m-u_i$ and $u^\prime+u^{\prime\prime}_j=p^su_m-u_j$. Summing all such terms implies
\begin{equation}
p^n\sigma(\alpha_{s-1})\equiv (f_{ij})\alpha_s\mod p^{s}
\end{equation}
and $p^n(f_{ij})^{-1}\equiv \alpha_s\sigma(\alpha_{s-1})^{-1}\mod p^{s-n}$.

Similar calculation as \cite{Katz1985} and congruence relation $D(B_{u^\prime})\equiv 0\mod p^s$ imply 
\begin{equation}
D(\alpha_s)\equiv (\nabla(D)_{ij})\alpha_s\mod p^s.
\end{equation}

If the coefficient of the vertex $u_0$ is zero, we regard $a_I$ as formal variables and the universal hypersurface family. Then we can prove the result on an open subset of $S$ and the $p$-adic limit formulas holds on the open subset. Since Vlasenko proved the congruence relations in Theorem \ref{cong} without any constraints on the coefficients, the $p$-adic limits always exist. So the limits coincide with Frobenius matrices and connection matrices because they are equal restricted to an open subset of $S$. 

Now we state the proof without assuming Conjecture \ref{key}. We claim $p^{l(n-1)}P(U_{\leq n-1})\subset p^{ln}H^n_{DR}(R_\infty[[t_1,\cdots, t_n]]/R_\infty)$. Applying Katz's argument of extension of scalars, we only need to prove this when $R$ is the Witt vectors of a perfect field. The Frobenius action on $U_{\leq n-1}$ divides $p^{n-1}$. So there exists $\sigma^{-1}$-linear map $\tilde{F}$ on $U_{\leq n-1}$ such that $\tilde{F}F=F\tilde{F}=p^{n-1}$. On the other hand, the Frobenius action on each element in $H^n_{DR}(R_\infty[[t_1\cdots t_n]]/R_\infty)$ has a factor $p^n$. So $p^{n-1}P(U_{\leq n-1})=P(p^{n-1}U_{\leq n-1})=P(F\tilde{F}U_{\leq n-1})\subset p^nH^n_{DR}(R_\infty[[t_1,\cdots, t_n]]/R_\infty)$ and $l$ iterations give $p^{l(n-1)}P(U_{\leq n-1})\subset p^{ln}H^n_{DR}(R_\infty[[t_1,\cdots, t_n]]/R_\infty)$. Multiplying both (\ref{image1}) and (\ref{image2}) by $p^{l(n-1)}$, we get
\begin{equation}
p^{l(n-1)} F(\omega_i^{(\sigma)})=p^{l(n-1)}\sum_j f_{ij} \omega_j \mod  p^{ln}\text{  in  } H^n_{DR}(R_\infty[[t_1,\cdots, t_n]]/R_\infty)
\end{equation}
and 
\begin{equation}
p^{l(n-1)} \nabla(D)(\omega_i)= p^{l(n-1)}\sum_j {\nabla(D)}_{ij} \omega_j \mod p^{ln} \text{  in  } H^n_{DR}(R_\infty[[t_1\cdots t_n]]/R_\infty).
\end{equation}
Let $s=nl$, similar congruence relation as (\ref{congruence1}) still holds for $k\leq s$
\begin{equation}
p^{n+l(n-1)}\sigma({A_{{u}_i^\prime}})\equiv  p^{l(n-1)}\sum_j f_{ij} A_{u_j^{\prime\prime}}\mod p^k
\end{equation}
with $p({{u}_i^\prime}+u_i)=p^kv=u_j^{\prime\prime}+u_j$ and $v>0$. The same argument shows
\begin{equation}
p^{n+l(n-1)}\sigma(\alpha_{s-1})\equiv p^{l(n-1)}(f_{ij})\alpha_s\mod p^{s}
\end{equation}
and 
\begin{equation}
p^{l(n-1)}D(\alpha_s)\equiv p^{l(n-1)}(\nabla(D)_{ij})\alpha_s\mod p^s.
\end{equation}
Dividing both sides by $p^{l(n-1)}$ and letting $l\to \infty$, we see that the subsquence $\alpha_s\sigma(\alpha_{s-1})^{-1}$ and $D(\alpha_{s})(\alpha_s)^{-1}$ converges to the Frobenius matrix and connection matrix. 
\end{proof}
\begin{rem}
\label{kernel}
If $U_{\leq n-1}$ is contained in the the kernel of formal expansion map, then it is exactly the kernel. We only need to show the restriction of expansion map on $H^0(X, \Omega^n_{X/S}(Y))\otimes R_\infty\to H^n_{DR}(R_\infty[[t_1\cdots t_n]]/R_\infty)$ is injective. This can be proved by similar argument in the proof and invertibility of $\alpha_s$. 
\end{rem}
\begin{rem}
The proof also gives a weaker version of the second and third congruence relations in Vlasenko's Theorem \ref{cong}. The first congruence relation $\alpha_s\equiv \alpha_1\cdot \sigma(\alpha_{s-1})\mod p$ can also be proved geometrically using the argument in Theorem \ref{localHW} and Corollary \ref{toricHW}. We can consider the $s$-iterated Hasse-Witt operation $H^{n-1}(Y_0^{(p^s)},\OX_{Y_0^{(p^s)}})\to H^{n-1}(Y_0, \OX_{Y_0})$. Using similar commutative diagram \ref{exactsequence} with the first vertical map $L^{-1}\to L^{-1}$ being replaced by the composition $L^{-1}\to L^{-p^s}\to L^{-1}$ with $\xi\mapsto \xi^{p^s}\cdot f^{p^s-1}$, we can see the matrix for $s$-iterated Hasse-Witt operation is given by $\alpha_s\mod p$. Hence $\alpha_s\equiv \alpha_1\cdot \sigma(\alpha_{s-1})\mod p$.
\end{rem}

\subsection{Unit root of toric Calabi-Yau hypersurfaces and periods}
\label{CYunitroot}
Now we discuss the relation between unit roots of zeta functions and period integrals for toric Calabi-Yau hypersurfaces. Let $f$ be a Laurent series defining toric Calabi-Yau hypersurfaces. For the sake of simplicity, let $a_0=1$ be the constant term (or the coefficient of interior point of $\Delta$). The unique holomorphic period integral at the special solution-1 point or "large complex structure limit" is $I_\gamma=$ the constant term of the expansion
\begin{equation}
{1\over f} ={1\over 1+\sum_{u_I\neq 0}a_{I}t^{u_I}}=1+\sum_k (-1)^k (\sum_{u_I\neq 0}a_{I}t^{u_I})^k.
\end{equation} 
It can be written as formal power series of $a_I$ with constant term being $1$ and denoted by $P(a_I)$. Then 
\begin{equation}
\alpha_s\equiv  \leftidx{^{(p^s-1)}} (P({a_I}))\mod p^s
\end{equation}
because of the congruence 
\begin{equation}
{p^s-1 \choose k_1,k_2,\cdots, k_l, p^s-1-k}\equiv (-1)^k{k \choose k_1,k_2,\cdots, k_l}\mod p^s.
\end{equation}
Then 
\begin{equation}
\alpha_s\sigma(\alpha_{s-1})^{-1}\equiv  {\leftidx{^{(p^s-1)}} (P({a_I}))\over  \leftidx{^{(p^{s-1}-1)}} (P)(\sigma({a_I}))}\mod p^s
\end{equation}
according to Vlasenko's congruences without any geometric constraints. So the $p$-adic limit $P(a_I)\over P(\sigma(a_I))$ exists in $R_\infty$ and equal to the Frobenius matrix.
We can fix $\sigma(a_I)=a_I^p$. Then the formal power series
\begin{equation}
g(a_I)={P(a_I)\over P(a_I^p)}
\end{equation}
has $p$-adic limit in $\varprojlim_{s\to \infty} \mathbb{Z}_p[a_{I_1}\cdots a_{I_N}, \alpha_1^{-1}]/p^s\mathbb{Z}_p[a_{I_1}\cdots a_{I_N}, \alpha_1^{-1}]$ and it satisfies Dwork congruences
\begin{equation}
g(a_I)\equiv {\leftidx{^{(p^s-1)}} (P({a_I}))\over  \leftidx{^{(p^{s-1}-1)}} (P)(\sigma({a_I}))}\mod p^s.
\end{equation}
Especially it is related to Hasse-Witt matrix by
\begin{equation}
{P(a_I)\over P(a_I^p)}\equiv \leftidx{^{(p-1)}} (P({a_I}))\mod p.
\end{equation}
Let $q=p^r$ and $a_I\in \mathbb{F}_q$ defining a smooth Calabi-Yau variety $Y_0$ over $\mathbb{F}_q$. Assume the Hasse-Witt matrix $\leftidx{^{(p-1)}} (P({a_I}))\mod p$ is not zero. Then there exist exactly one $p$-adic unit root in the factor of zeta function of $Y_0$ corresponding to Frobenius action on $H^{n-1}_{cris}(Y_0)$. It is given by 
\begin{equation}
g(\hat{a}_I)g(\hat{a}_I^p)\cdots g(\hat{a}_I^{p^{r-1}})
\end{equation}
with $\hat{a}_I$ being the Teichm\"uller lifting under $\sigma$. For example, if $a_I$ has lifting as an integer, then $\hat{a}_I=\lim_{s\to \infty}a_I^{p^s}$. 
\begin{rem}
In \cite{vlasenko}, the following result about unit roots is proved. When $S_0=\spec(\mathbb{F}_q)$ and $Y_0$ is a smooth hypersurface in $\PP^n$, let $\Phi=F\cdot \sigma(F)\cdots \sigma^{r-1}(F)$. Then the eigenvalues of $\Phi$ are unit roots of zeta-function of $Y_0$. The conjecture proved above implies that the multiplicities of unit roots are also equal. The proof in \cite{vlasenko} uses Stienstra's result on formal groups \cite{stienstra, stienstra1}. See also \cite{jdyu} for the unit root formula for Dwork family using formal groups following Stienstra. It might be possible to give a proof of the conjecture by this approach.
\end{rem}
\section{Frobenius matrices for Calabi-Yau hypersurfaces}
\label{flagF}
Now we discuss the algorithm of Frobenius matrix of Calabi-Yau hypersurfaces in terms of local expansion. Let $X$ be a smooth Fano variety over $S$ with ample line bundle $L$. Let $f\in H^0(X, L)$ define a smooth hypersurface $Y$ in $X$. Let $\omega_i$ be a basis of $H^0(X, L\otimes K_X^{-1})$. This induces a basis of $H^0(Y, K_Y)$ via adjunction formula. The Hasse-Witt matrix under this basis in terms of local coordinate is given by the algorithm in Section \ref{local}. Fix a section $p\colon S\to X$ on ambient space $X$ such that $Y$ is away from $p$. Let $(t_1, \cdots, t_n)$ be the formal coordinate of $X$ at $p$. The proof of Theorem $\ref{main2}$ depends on $\omega_i$ having the following form in local expansion.  There is a trivialization $\xi$ of $L$ along $p$ such that $w_i\over \xi$ has the form ${t^{u_i}dt_1\wedge \cdots \wedge dt_n\over t_1\cdots t_n}$. The matrix $(\alpha_s)_{ij}$ is defined to be the coefficients of $t^{p^su_j-u_i}$ in $f^{p^s-1}\over \xi^{p^s-1}$. This applies to $L=K_X^{-1}$ with trivialization $\xi=(dt_1\wedge \cdots \wedge dt_n)^{-1}$. So we have the following
\begin{prop}
\label{CYFrob}
Let $f=g(t)(dt_1\wedge \cdots \wedge dt_n)^{-1}$ and $\alpha_s=$ the coefficient of $(t_1\cdots t_n)^{p^s-1}$ in the local expansion of $g^{p^s-1}$. Then similar congruence relations in Vlasenko's theorem (Theorem \ref{cong}) still stand
\begin{enumerate}
\item For $s\geq 1$, $$\alpha_s\equiv \alpha_1\cdot (\alpha_{s-1})^p\mod p.$$
\item Under the condition $HW(n-1)$ and $g(0)\neq 0$, we have
$$ \alpha_{n(s+1)}\cdot \sigma(\alpha_{n(s+1)-1})^{-1}\equiv \alpha_{ns}\cdot \sigma(\alpha_{ns-1})^{-1}\mod p^{s-n}.$$
\item Under the condition of (2), for any derivation $D\colon R\to R$, we have
$$D(\sigma(\alpha_{n(s+1)}))\cdot \sigma(\alpha_{n(s+1)})^{-1}\equiv D(\sigma(\alpha_{ns}))\cdot \sigma(\alpha_{ns})^{-1}\mod p^{s}.$$
\end{enumerate}
The $p$-adic limit $\alpha_{ns}\cdot \sigma(\alpha_{ns-1})^{-1}$ gives the Frobenius action on the unit root part $U_0$ of $H^{n-1}_{cris}(Y_0)$ under the basis induced by residue map.
\end{prop}

\subsection{Unit root of Calabi-Yau hypersurfaces in $G/P$}
Now we discuss the algorithm for Frobenius matrix of the unit root part of Calabi-Yau hypersurfaces in $X=G/P$. Consider the affine chart $\mathbb{A}^n$ on the Bott-Samelson desingularization of $G/P$ in Section \ref{flagv}. This induces a torus chart $\mathbb{G}_m^n$ on $G/P$. We consider the formal polydisc $R_\infty[[t_1, \cdots, t_n]][{1\over t_1},\cdots, {1\over t_n}]$ instead of $R_\infty [[t_1, \cdots, t_n]]$ in the formal expansion map in the proof of Theorem \ref{main2}. The same method gives the algorithm of the unit root part of the Frobenius action.
\begin{thm}
\label{main3}
Let $f\in H^0(X, K_X^{-1})$ be an anticanonical form defining an smooth Calabi-Yau hypersurface $Y$. Assume $f$ has the form $f=g(t)(dt_1\wedge \cdots \wedge dt_n)^{-1}$ with $g(t)\in R[t_1, \cdots, t_n][{1\over t_1},\cdots, {1\over t_n}]$ in the torus chart as above. Assume the hypersurface $Y$ is away from the image of $(t_1, \cdots, t_n)=0$. Let $\alpha_s$ be the coefficient of $(t_1\cdots t_n)^{p^s-1}$ in the local expansion of $g^{p^s-1}$. The Hasse-Witt matrix is given by $\alpha_1\mod p$. The same congruence relations and Frobenius matrix in proposition \ref{CYFrob} holds.
\end{thm}
\begin{proof}
The function $g(t)=t_1^{-k_1}\cdots t_n^{-t_k}g^\prime(t)$ for some $k_i\geq 0$ and $g^\prime(t)\in R[t_1, \cdots, t_n]$ does not vanish at $(t_1, \cdots, t_n)=0$. This is because the torus chart extends to an map on $\mathbb{A}^n$ and $k_i$ are the multiplicity of exceptional divisor. So the index $p^kv$ appearing in the proof of Theorem \ref{main2} still has positive components. The rest of the proof is the same as toric case. 
\end{proof}
In Section \ref{flagv}, the period integral for the Calabi-Yau hypersurfaces in $G/P$ is reduced to similar algorithm on the torus chart. So the same argument in Section \ref{CYunitroot} implies similar relations between periods and the unit root of zeta-function of Calabi-Yau hypersurface defined on finite fields. 

\begin{rem}
We require the non-vanishing condition in Theorem \ref{main2} and \ref{main3} to discuss the local expansion map. But the definition of matrix $\alpha_s$ and congruence relations do not require this condition. So there might be a proof for general cases not depending on local expansions.
\end{rem}

\bibliography{reference3.bib}

\begin{thebibliography}{10}

\bibitem{adolphson2014}
A.~Adolphson and S.~Sperber.
\newblock Hasse invariants and mod $p$ solutions of {A}-hypergeometric systems.
\newblock {\em Journal of Number Theory}, 142:183--210, 2014.

\bibitem{adolphson2016}
A.~Adolphson and S.~Sperber.
\newblock {A}-hypergeometric series and the {Hasse}--{Witt} matrix of a
  hypersurface.
\newblock {\em Finite Fields and Their Applications}, 41:55--63, 2016.

\bibitem{Brion2005}
M.~Brion.
\newblock Lectures on the geometry of flag varieties.
\newblock In {\em Topics in cohomological studies of algebraic varieties},
  pages 33--85. Springer, 2005.

\bibitem{Brion}
M.~Brion and S.~Kumar.
\newblock {\em Frobenius splitting methods in geometry and representation
  theory}, volume 231.
\newblock Springer Science \& Business Media, 2007.

\bibitem{Brion2003}
M.~Brion and V.~Lakshmibai.
\newblock A geometric approach to standard monomial theory.
\newblock {\em Representation Theory of the American Mathematical Society},
  7(25):651--680, 2003.

\bibitem{candelas}
P.~Candelas, X.~de~la Ossa, and F.~Rodriguez-Villegas.
\newblock {Calabi-Yau Manifolds Over Finite Fields, I}.
\newblock {\em arXiv preprint hep-th/0012233}, 2000.

\bibitem{Dwork1}
B.~Dwork.
\newblock A deformation theory for the zeta function of a hypersurface.
\newblock In {\em Proc. of the Int. Congress of Math., Stockholm}, pages
  247--259, 1962.

\bibitem{Dwork2}
B.~Dwork.
\newblock {$p$}-adic cycles.
\newblock {\em Inst. Hautes \'Etudes Sci. Publ. Math.}, (37):27--115, 1969.

\bibitem{Dwork3}
B.~Dwork.
\newblock On {H}ecke polynomials.
\newblock {\em Invent. Math.}, 12:249--256, 1971.

\bibitem{Fulton}
W.~Fulton.
\newblock A fixed point formula for varieties over finite fields.
\newblock {\em Mathematica Scandinavica}, 42(2):189--196, 1978.

\bibitem{HLZ}
A.~Huang, B.~H. Lian, and X.~Zhu.
\newblock Period integrals and the {Riemann--Hilbert} correspondence.
\newblock {\em Journal of Differential Geometry}, 104(2):325--369, 2016.

\bibitem{Katz1}
N.~Katz.
\newblock Travaux de {D}work.
\newblock pages 167--200. Lecture Notes in Math., Vol. 317, 1973.

\bibitem{Katz}
N.~M. Katz.
\newblock Algebraic solutions of differential equations ({$p$-curvature and the
  Hodge filtration)}.
\newblock {\em Inventiones mathematicae}, 18(1):1--118, 1972.

\bibitem{Katz1984}
N.~M. Katz.
\newblock Expansion-coefficients as approximate solution of
  differential-equations.
\newblock {\em Ast{\'e}risque}, (119):183--189, 1984.

\bibitem{Katz1985}
N.~M. Katz.
\newblock Internal reconstruction of unit-root {$ F $}-crystals via
  expansion-coefficients. {With an appendix by Luc Illusie}.
\newblock In {\em Annales scientifiques de l'{\'E}cole Normale Sup{\'e}rieure},
  volume~18, pages 245--285. Elsevier, 1985.

\bibitem{Thomas}
A.~Knutson, T.~Lam, and D.~E. Speyer.
\newblock Projections of {Richardson} varieties.
\newblock {\em Journal f{\"u}r die reine und angewandte Mathematik (Crelles
  Journal)}, 2014(687):133--157, 2014.

\bibitem{lam2017}
T.~Lam and N.~Templier.
\newblock The mirror conjecture for minuscule flag varieties.
\newblock {\em arXiv preprint arXiv:1705.00758}, 2017.

\bibitem{marsh2013}
R.~Marsh and K.~Rietsch.
\newblock The {B}-model connection and mirror symmetry for {Grassmannians}.
\newblock {\em arXiv preprint arXiv:1307.1085}, 2013.

\bibitem{rietsch2008}
K.~Rietsch.
\newblock A mirror symmetric construction of {$qHT_*(G/P)(q)$}.
\newblock {\em Advances in Mathematics}, 217(6):2401--2442, 2008.

\bibitem{rietsch2012}
K.~Rietsch.
\newblock A mirror symmetric solution to the quantum toda lattice.
\newblock {\em Communications in Mathematical Physics}, 309(1):23--49, 2012.

\bibitem{stienstra}
J.~Stienstra.
\newblock Formal group laws arising from algebraic varieties.
\newblock {\em American Journal of Mathematics}, 109(5):907--925, 1987.

\bibitem{stienstra1}
J.~Stienstra.
\newblock Formal groups and congruences for {L}-functions.
\newblock {\em American Journal of Mathematics}, 109(6):1111--1127, 1987.

\bibitem{vlasenko}
M.~Vlasenko.
\newblock Higher {Hasse}--{Witt} matrices.
\newblock {\em arXiv preprint arXiv:1605.06440}, 2016.

\bibitem{jdyu}
J.-D. Yu.
\newblock Variation of the unit root along the {Dwork} family of
  {Calabi}--{Yau} varieties.
\newblock {\em Mathematische Annalen}, 343(1):53--78, 2009.

\end{thebibliography}
\bibliographystyle{abbrv}
\end{document}